\documentclass{article}%

\usepackage{amsmath}%
\usepackage{amsfonts}%
\usepackage{amssymb}%
\usepackage{graphicx}
\usepackage{bbm}
\usepackage{palatino}
\usepackage{fullpage}
\usepackage[english]{babel}
\usepackage{booktabs}
\usepackage{tikz}
\usetikzlibrary{plotmarks}
\tikzset{
    scale plot marks/.is choice,
    scale plot marks/false/.code={
        \def\pgfuseplotmark##1{\pgftransformresetnontranslations\csname pgf@plot@mark@##1\endcsname}
    },
    scale plot marks/true/.style={},
    scale plot marks/.default=true
}
\usepackage{subcaption}
\usepackage[boxed, linesnumbered]{algorithm2e}
\usepackage[margin=2cm]{caption}
\usepackage{dsfont}
\usepackage[bookmarksnumbered=true]{hyperref}

\newtheorem{theorem}{Theorem}

\newtheorem{example}{Example}

\newtheorem{lemma}[theorem]{Lemma}

\newtheorem{proposition}[theorem]{Proposition}
\newtheorem{remark}{Remark}
\newtheorem{assumption}{Assumption}

\newenvironment{proof}[1][Proof]{\textbf{#1.} }{\ \rule{0.5em}{0.5em}}

\numberwithin{equation}{section}

\newcommand{\Prob}{\mathbb{P}}
\newcommand{\Exp}{\mathbb{E}}

\begin{document}
\title{\textbf{First-passage time asymptotics over moving boundaries\\ for random walk bridges}\vspace{-1ex}}
\author{Fiona Sloothaak \and Vitali Wachtel \and Bert Zwart}
\maketitle

\begin{abstract}
We study the asymptotic tail probability of the first-passage time over a moving boundary for a random walk conditioned to return to zero, where the increments of the random walk have finite variance. Typically, the asymptotic tail behavior may be described through a regularly varying function with exponent $-1/2$, where the impact of the boundary is captured by the slowly varying function. Yet, the moving boundary may have a stronger effect when the tail is considered at a time close to the return point of the random walk bridge. In the latter case, a phase transition appears in the asymptotics, of which the precise nature depends on the order of distance between zero and the moving boundary.
\end{abstract}

\section{Introduction}
The asymptotic behavior of random walks has long been an extremely popular topic in probability. In 1951, Donsker~\cite{Donsker1951} shows that a suitably rescaled stable random walk converges to a Brownian motion. Many extensions have been studied over the years, such as generalizations to random walks in the domain of attraction of a stable law~\cite{Skorokhod1957} and additional conditioning properties. For example, an invariance principle was shown for random walk bridges in~\cite{Liggett1968}, whereas~\cite{Iglehart1974,Bolthausen1976,Doney1985} developed invariance principles for random walks conditioned to stay positive. Recently, these two types of conditioning have been combined in~\cite{Caravenna2013} to an invariance principle for random walk bridges conditioned to stay positive over the entire interval.

A natural question that arises is whether and how these results extend to moving boundaries. That is, how does the random walk behave asymptotically, conditioned it stays above a boundary sequence that is not necessarily zero or even constant? This topic, as well as the closely related first-passage asymptotics, has been addressed in~\cite{Novikov1981,Novikov1983,Greenwood1984,Greenwood1987,DenisovShneer2013,Aurzada2016,Denisov2016} and many more.

In this paper, we include a moving boundary in a particular random walk bridge setting. More precisely, we consider a random walk bridge with increments that have zero mean and finite variance. The purpose is to derive the asymptotic tail of the first-passage time over a moving boundary for this random walk bridge. We stress that we are only considering the tail for all times that are well before the random walk bridge returns to zero. In other words, we extend a random walk bridge in the Brownian setting to stay above a moving boundary over part of its interval.

Besides this problem being of intrinsic interest, our inspiration comes from a seemingly unrelated area: cascading failure models. These models are used to describe systems of interconnected components where failures possibly trigger subsequent failures of other components. A typical reliability measure in such problems is the probability that the number of failures exceeds a certain threshold. Analytic results are obtained for particular settings~\cite{Dobson2004,Sloothaak2016}, but allow for limited generalizations. It turns out that this problem has an equivalent random-walk bridge representation, where the objective translates to the probability that the first-passage time over a moving boundary exceeds the threshold. In Section~\ref{sec:MainResults} we demonstrate this relation in detail.

Clearly, the asymptotic behavior of this first-passage time depends on the boundary sequence. We consider all boundary sequences that are within square-root order from zero, and hence relatively not too far from zero with respect to time. Our results distinguishes between two regimes. The first concerns thresholds that are significantly smaller than the time that the random walk returns to zero, whereas the second considers thresholds close to the return point. In the first case, the objective  can be described asymptotically through a regularly varying function with exponent $-1/2$. When the boundary sequence satisfies certain additional conditions, as explained in e.g.~\cite{Denisov2016}, the objective has a power-law decay with a preconstant that can be interpreted in a probabilistic way. However, depending on how close the boundary remains to zero, a phase transition possibly occurs when the threshold is close to the return point of the random walk bridge. This intriguing phenomenon reflects the strong dependence on the boundary: the effect may not solely be captured in a slowly varying function, but can affect the behavior much more drastically.

The paper is organized as follows. In Section~\ref{sec:Preliminaries} we state our assumptions, and point out known results that are used throughout this paper. Our main results are presented in Section~\ref{sec:MainResults}. In Section~\ref{sec:ThresholdSufficientlyFarFromReturnPoint} we prove the result in case that the threshold is sufficiently far from the return point, while the proof in the other case is given in Section~\ref{sec:Threshold close to return point}. In particular, the proof in the latter case requires a result on the uniform convergence of an unconditioned random walk while staying above a moving boundary, which we state and prove in~Section~\ref{secsec:DensityOfRWAtTheThreshold}.

\section{Preliminaries}
\label{sec:Preliminaries}
Before presenting our main results, we first introduce some notation, state our assumptions and point out the consequences.

\subsection{Notation}
Let $X_i$, $i\geq 1$ be independent, identically distributed random variables with $\Exp X_i=0$ and $\Exp X_i^2=1$ for all $i\geq 1$. Define the random walk
\begin{align*}
\begin{array}{ll}
S_m:= \sum_{i=1}^m X_i, & m\geq 1.
\end{array}
\end{align*}

\noindent
We refer to $\{g_i\}_{i \in \mathbb{N}}$ as the boundary sequence. Define the stopping time
\begin{align*}
\tau_g := \min\{i \geq 1 : S_i \leq g_i \},
\end{align*}
i.e.~the first-passage time of the random walk over the (moving) boundary. In case that $g_i=x$ for all $i \geq 1$, we write $T_x:=\tau_g$ for the stopping to emphasize the constant boundary.

Finally, we use some notation throughout this paper to classify the order of magnitude. We write $a_n=o(b_n)$ if $\limsup_{n \rightarrow \infty} a_n/b_n =0$ and $a_n=O(b_n)$ if $\limsup_{n \rightarrow \infty} a_n/b_n< \infty$. Similarly, we write $a_n=\omega(b_n)$ if $\lim_{n \rightarrow \infty} b_n/a_n =0$ and $b_n=\Omega(b_n)$ if $\limsup_{n \rightarrow \infty} b_n/a_n< \infty$. Finally, we write $a_n = \Theta(b_n)$ if both $a_n=O(b_n)$ and $a_n = \Omega(b_n)$, and denote $a_n \sim b_n$ if $\lim_{n\rightarrow\infty} a_n/b_n=1$.

\subsection{Assumptions and properties}
\label{secsec:Assumptions}
First, we make some assumptions on the increments.
\begin{assumption}
The increments of the random walk are independent and identically distributed with mean zero and variance one. Additionally, we assume the that the law of the increments has a density $f(\cdot)$ (almost everywhere) and that there exists a $n_0$ such that $f_{n_0}(\cdot)$, the density corresponding to $S_{n_0}$, is bounded (almost everywhere).
\label{ass:Increments}
\end{assumption}

Note that every random walk with increments with mean zero and finite variance can be rescaled to satisfy Assumption~\ref{ass:Increments}, i.e.~we cover all random walks with increments that fall in the normal domain of attraction of the normal distribution. We point out that the boundedness requirement on the density function of the random walk for some $n_0$ in Assumption~\ref{ass:Increments} is the necessary and sufficient for a uniform convergence between the scaled density of the position of the random walk towards the standard normal density~\cite[p.~198]{Petrov1975}. Specifically, let $\phi(\cdot)$ and $\Phi(\cdot)$ denote the density function and the distribution function of a standard normal random variable, respectively. Then,
\begin{align}
\lim_{n\rightarrow \infty} \sup_{x \in \mathbb{R}} \left| \Prob\left( \frac{S_n}{\sqrt{n}} \leq x \right) - \Phi(x) \right| = 0, \hspace{1cm} \lim_{n\rightarrow \infty} \sup_{x \in \mathbb{R}} \left| \sqrt{n} f_n(\sqrt{n} x) - \phi(x) \right| = 0.
\label{eq:RandomWalkConvergenceProperties}
\end{align}

\noindent
Next, we assume that the boundary sequence $\{g_i\}_{i \in \mathbb{N}}$ does not move too far from zero.
\begin{assumption}
The boundary sequence $\{g_i\}_{i \in \mathbb{N}}$ satisfies
\begin{align}
|g_i|= o(\sqrt{i})
\label{eq:ConditionBoundary1}
\end{align}
for all $i \geq 1$, and 
\begin{align}
\Prob\left( \tau_g > n \right)>0, \hspace{0.5cm} \forall n \geq 1.
\label{eq:ConditionBoundary2}
\end{align}
\label{ass:Boundary}
\end{assumption}

Under Assumption~\ref{ass:Boundary}, it is known that the position of the rescaled random walk, conditioned it stays above the boundary, converges to a Rayleigh distribution~\cite{Denisov2016}. To be precise, as $n \rightarrow \infty$,
\begin{align}
\Prob\left(S_n > g_n + v \sqrt{n} | \tau_g > n \right) \sim e^{-v^2/2}, \hspace{0.5 cm} \forall v\geq 0.
\label{eq:ConditionedRWDensity}
\end{align}
The first-passage time itself has a regularly varying tail,
\begin{align}
\Prob(\tau_g > n) \sim \sqrt{\frac{2}{\pi}} \frac{L_g(n)}{\sqrt{n}},
\label{eq:ConditionedRWTauStoppingTime}
\end{align}
where $L_g(\cdot)$ is a positive, slowly varying function. This slowly varying function has a probabilistic interpretation:
\begin{align}
L_g(n) = \Exp(S_n-g_n ; \tau_g >n) \sim \Exp(-S_{\tau_g} ; \tau_g \leq n) \in (0,\infty).
\end{align}
The literature offers many discussions and results for which this slowly varying function converges to a finite constant $L_g(\infty):= \lim_{n \rightarrow \infty} L_g(n)$. We note that in case that $L_g(\infty) < \infty$ exists, the slowly varying term can hence be replaced by the constant $\Exp(-S_{\tau_g})$. We refer the reader to~\cite{Denisov2016} for a thorough discussion of this issue. 

\begin{remark}\normalfont
To the best of our knowledge~\cite{Denisov2016} provide the least strong conditions that allow for the existence of a finite $L_g(\infty)$. These conditions are a bit cumbersome, and alternatively, we mention two easily checked cases here that are known from former literature.

In~\cite{Greenwood1987}, they show that if the boundary sequence $g_n, n\geq 1$ is non-increasing and concave, then $L_g(\infty)$ exists and 
\begin{align*}
\sum_{n=1}^\infty \frac{-g_n}{n^{3/2}} < \infty \Longleftrightarrow L_g(\infty) = \Exp(-S_{\tau_g}) \in (0,\infty).
\end{align*}
Particularly, this holds for all finite constant boundaries. In Theorem~5 of~\cite{Wachtel2016a}, the concavity condition is relaxed, but a stronger summability condition is required. Specifically,it is shown that if $g_n, n\geq 1$ is non-increasing,
\begin{align*}
\sum_{n=1}^\infty \frac{\log^{1/2} n}{n^{3/2}}(-g_n) < \infty \Longrightarrow L_g(\infty) = \Exp(-S_{\tau_g}) \in (0,\infty).
\end{align*}
\end{remark}

In this paper, we consider a random walk that returns to zero at time $n$. The objective is to derive the asymptotic behavior of the probability that this random walk stay above a moving boundary over part of its interval. That is, given that the random walk returns to zero at time $n$, what is the asymptotic probability of the random walk staying above the moving boundary up to time $k:=k_n$? We refer to $k$ as the threshold and assume that it is at least $\omega(1)$ distance from both zero and $n$.

\begin{assumption}
\label{ass:Threshold}
The threshold $k:=k_n$ satisfies both $k \rightarrow \infty$ and $n-k \rightarrow \infty$ as $n\rightarrow \infty$. 
\end{assumption}

\section{Main results}
\label{sec:MainResults}
We distinguish between two cases: one where the threshold is not too close to the point of return of the random walk bridge, and one where it is.
\begin{theorem}
Suppose $\lim_{n \rightarrow \infty} k/n <1$. Then, as $n \rightarrow \infty$,
\begin{align}
\Prob\left( \tau_g >k |S_n =0 \right) \sim
\sqrt{\frac{2}{\pi}} L_g(k) \sqrt{\frac{n-k}{n}} k^{-1/2}.
\end{align}
\label{thm:MainResultProportional}
\end{theorem}

Note that when $k$ is not too close to the boundary, the impact of the boundary is completely captured by the slowly varying function. When $k$ moves closer to $n$, the behavior of the boundary turns more relevant and possibly results to a change in asymptotics.

\begin{theorem}
Suppose $k=n-o(n)$. Additional to the Assumption~\ref{ass:Boundary}, suppose there exists a $\epsilon \in (0,1)$ such that 
\begin{align}
\sup_{j \in [(1-\epsilon)k,k]} |g_j-g_k| \leq \alpha(\epsilon) |g_k|,
\label{eq:AdditionalAssumption}
\end{align}
for every large enough $n$, with $\alpha(\epsilon) \rightarrow 0$ as $\epsilon \downarrow 0$. Then, as $n \rightarrow \infty$,
\begin{align}
\Prob\left( \tau_g >k |S_n =0 \right) \sim \left\{ \begin{array}{ll}
\sqrt{\frac{2}{\pi}} L_g(k) \frac{\sqrt{n-k}}{k}, &\textrm{if } |g_k| = o(\sqrt{n-k}), \\
\sqrt{\frac{2}{\pi}} L_g(k) \gamma\left(\frac{|g_k|}{\sqrt{n-k}}\right) \frac{\sqrt{n-k}}{k}, &\textrm{if } |g_k| = \Theta(\sqrt{n-k}), \\
2 L_g(k) \frac{|g_k|}{k}, &\textrm{if } |g_k| = \omega(\sqrt{n-k}), g_k<0,\\
\end{array} \right.
\label{eq:MainTheoremBoundary}
\end{align}
where 
\begin{align}
\begin{split}
\begin{array}{ll}
\gamma(y) := e^{-\frac{y^2}{2}} - y \int_{x=y}^\infty e^{-\frac{x^2}{2}} \, dx. 
\end{array}
\end{split}
\label{eq:DefEtaGamma}
\end{align}
\label{thm:MainResultBoundary}
\end{theorem}

A typical example that is covered by this framework is when $g_i=-i^\alpha, i\in \mathbb{N}$ with $\alpha<1/2$. The additional assumption~\eqref{eq:AdditionalAssumption} is merely technical: it ensures that the boundary does not fluctuate too much as it moves closer to the threshold. That is, for every $\epsilon>0$ there is a value $\alpha(\epsilon)< \infty$ such that the boundary values does not fluctuate more than $2\alpha(\epsilon)g_k$ in the interval $[(1-\epsilon)k,k]$ for large enough $n$. Particularly, this implies that $\epsilon \in (0,1)$ can be chosen small enough such that $\alpha(\epsilon)<1$, and hence the boundary sequence at time $[(1-\epsilon)k,k]$ has the same sign (either positive, negative or zero). Cases where the boundary sequence strongly oscillates close to the threshold are thus excluded from our framework.

The phase transition that appears in Theorem~\ref{thm:MainResultBoundary} reflects the strong influence of the boundary sequence in this case. It might not be captured solely by the slowly varying function, but can have a much stronger effect. Furthermore, this effect is only influenced by the behavior of the boundary sequence close to the threshold. This observation is best-explained by our approach. We track the position of a random walk at time $k$, conditioned that it stays above the moving boundary till that point. Then we evaluate how likely a reversed random walk moving back from time $n$ can reach that point. The reversed random walk will converge uniformly to a normal density, and is therefore likely to stay within $\sqrt{n-k}=o(\sqrt{k})$ distance from zero. When $g_k<0$, those values are thus likely of order $\max\{\sqrt{n-k},|g_k|\}$ distance from the boundary. The phase transition is then a consequence of how likely the random walk staying above the boundary sequence can move to such values.

\begin{example}\normalfont
As pointed out in the introduction, Theorem~\ref{thm:MainResultProportional} and~\ref{thm:MainResultBoundary} can be applied to a seemingly unrelated problem in cascading failure models. In this example, we will describe a particular cascading failure model as in~\cite{Sloothaak2016}, and translate it to the random-walk bridge setting we consider in this paper.

Consider a system consisting of $n$ (indistinguishable) components. Each component has a limited capacity for the amount of load it can carry before it fails. The network is initially stable, in the sense that every component has sufficient capacity that exceeds the initial load. We assume that the difference between the initial loads and capacities, which we refer to as the \textit{surplus capacity}, are stochastic random variables that are independent and identically distributed with continuous distribution function $F(\cdot)$. In order to trigger a possible cascading failure effect, we include an initial disturbance that causes all components to be additionally loaded with $l_n(1)$. When the capacity of a component is exceeded by its load demands, that component fails. Every component failure causes (equal) additional loading on the remaining components, possibly triggering knock-on effects. We write $l_n(i)$ for the total load surge per component when $i-1$ components have failed, and assume this is a deterministic non-decreasing function. The cascading failure process continues until the capacities on the remaining components are sufficient to deal with the load increases.

A measure of system reliability is the number of component failures after the cascading failure process has ended, written by $A_n$. Since $F(\cdot)$ is continuous, it satisfies the identity~\cite{Sloothaak2016}
\begin{align*}
\Prob\left(A_n \geq k\right) = \Prob\left( U_{(i)}^n \leq F\left(l_n(i)\right), \hspace{.25cm} i=1,...,k \right),
\end{align*}
where $U_{(i)}^n$ denotes the $i$'th order statistic of $n$ uniformly distributed random variables with support $[0,1]$. In~\cite{Sloothaak2016}, they are interested which choices of $F(\cdot)$ and $l_n(\cdot)$ asymptotically exhibits power-law behavior for large thresholds $k$ as in Assumption~\ref{ass:Threshold}. In particular, they consider a setting where 
\begin{align}
F\left(l_n(i)\right) = \frac{\theta+i-1}{n}.
\label{eq:AffineCase}
\end{align}

Next, we show how this problem can be related to our random-walk bridge framework. Consider the random walk $S_n = i-\sum_{i=1}^n E_i$ where $(E_i)_{i\in \mathbb{N}}$ are independent identically exponentially distribution random variables with mean one. It is well-known that
\begin{align*}
\left(U_{(1)}^{n},U_{(2)}^{n},...,U_{(n)}^{n}\right) \overset{d}{=} \left(\frac{E_{1}}{n},\frac{\sum_{i=1}^2 E_i}{n},...,\frac{\sum_{i=1}^n E_i}{n} \, \bigg| \, \sum_{i=1}^{n+1} E_i=n \right).
\end{align*}
Then the probability that the number of component failures exceeds the threshold can be written as
\begin{align*}
\Prob(A_n \geq k) &= \Prob\left( U^n_{(i)} \leq \frac{\theta+i-1}{n}, \hspace{0.25cm} i=1,...,k \right)=  \Prob\left( S_i \geq 1-\theta, \hspace{0.25cm} i=1,...,k | S_{n+1} = 1 \right) \\
&\sim  \Prob\left( S_i \geq 1-\theta, \hspace{0.25cm} i=1,...,k | S_{n} = 0 \right) = \Prob\left( T_{1-\theta} >k \big| S_{n} = 0 \right).
\end{align*}
Theorem~\ref{thm:MainResultProportional} and~\ref{thm:MainResultBoundary} yields the result immediately. That is, as $n\rightarrow \infty$, we obtain 
\begin{align*}
\Prob(A_n \geq k) \sim \sqrt{\frac{2}{\pi}} L_{1-\theta}(k) \sqrt{\frac{n-k}{k n}},
\end{align*}
and since the boundary is constant,
\begin{align*}
\lim_{n \rightarrow \infty} L_{1-\theta}(k) \sim \Exp\left(-S_{T_{1-\theta}}\right) = -(1-\theta)+1=\theta.
\end{align*}
due to the memory-less property of exponentials.

Yet,~\eqref{eq:AffineCase} is a very specific case. In~\cite{Sloothaak2016}, they explore
for which perturbations the power-law behavior prevails. That is, if
\begin{align*}
F\left(l_n(i)\right) = \frac{\theta+i-1-g(i)}{n},
\end{align*}
which perturbations $g(\cdot)$ yield power-law behavior? The analytic approach used in~\cite{Sloothaak2016} allows for relatively limited generalizations. Theorem~\ref{thm:MainResultProportional} and~\ref{thm:MainResultBoundary} provide the answer to a much broader range of possible perturbations, and quantifies its effect on the prefactor in a probabilistic way. 
\end{example}

\section{Threshold sufficiently far from return point}
\label{sec:ThresholdSufficientlyFarFromReturnPoint}
We first consider the case where $\lim_{n \rightarrow \infty} k/n <1$ as in Theorem~\ref{thm:MainResultProportional}. Define the reversed random walk as
\begin{align}
\begin{array}{ll}
\tilde{S}_m = S_{n-m} , & m \geq 0,
\end{array}
\label{eq:DefinitionReversedRW}
\end{align}
where $S_n=0$ (and no condition is posed on the $S_0$). Consequently, $\tilde{S}_m$ obeys the same law as $-S_m$ for all $m\geq 0$. In the proof, we evaluate all events that a random walk conditioned to stay above the moving boundary meets  at time $k$ the reversed (unconditioned) random walk starting at time $k$. When $\lim_{n \rightarrow \infty} k/n <1$, we can use a direct approach to derive the asymptotic behavior.

\begin{proof}[Proof of Theorem~\ref{thm:MainResultProportional}]
Note that
\begin{align*}
\Prob\left(\tau_g > k | S_n = 0  \right) &= \int_{u=g_k}^\infty \Prob\left(\tau_g > k ; S_k \in du | S_n = 0  \right) = \frac{1}{f_n(0)} \int_{u=g_k}^\infty  \Prob\left(S_k \in du ; \tau_g > k \right) \tilde{f}_{n-k}(u)\\
&= \frac{\Prob(\tau_g>k)}{f_n(0)} \int_{u=g_k}^\infty  \Prob\left(S_k \in du | \tau_g > k \right) \tilde{f}_{n-k}(u),
\end{align*}
where $\tilde{f}_{n-k}(\cdot)$ is the density of the reversed random walk at time $n-k$.

The strategy is to bound the integral by two sums that coincide to the same expression as $n \rightarrow \infty$. Fix $\Delta >0$, and write
\begin{align*}
\Delta \mathbb{N}_0 := \{\Delta i : i \geq 0, i \in \mathbb{Z}\}.
\end{align*}
Since the reversed random walk $\tilde{S}_m$ is the same in distribution as $-S_m$, $m\geq 1$, it also satisfies~\eqref{eq:RandomWalkConvergenceProperties} and hence there is a uniform convergence to the normal density. Then,
\begin{align*}
\int_{u=g_k }^\infty  &\Prob\left(S_k \in du | \tau_g > k \right) \tilde{f}_{n-k}(u)  = \int_{v=g_k/\sqrt{k}}^\infty  \Prob\left(\frac{S_k}{\sqrt{k}} \in dv | \tau_g > k \right) \tilde{f}_{n-k}(v \sqrt{k}) \\
&\leq \sum_{z \in \Delta \mathbb{N}_0} \Prob\left(\frac{S_k}{\sqrt{k}} \in \left[\frac{g_k}{\sqrt{k}} + z,\frac{g_k}{\sqrt{k}} + z+\Delta\right) \big| \tau_g > k \right) \sup_{y \in \left[z, z+\Delta\right)} \tilde{f}_{n-k}( g_k + y \sqrt{k}) \\
&\leq \frac{1+o(1)}{\sqrt{2\pi (n-k)}} \sum_{z \in \Delta \mathbb{N}_0} \Prob\left(\frac{S_k}{\sqrt{k}} \in \left[\frac{g_k}{\sqrt{k}} + z,\frac{g_k}{\sqrt{k}} + z+\Delta\right) \big| \tau_g > k \right) \sup_{y \in \left[z, z+\Delta\right)} e^{-\frac{(g_k + y \sqrt{k})^2}{2(n-k)}}.
\end{align*}
Since $e^{-x^2}$ attains its maximum at $x=0$ and is decreasing for $x>0$, we see that the supremum over an interval $[x,x+\Delta)$ of this function is attained at $x$ for all $x\geq0$. Furthermore, since $|g_k|=o(\sqrt{k})$, we observe that $g_k + y \sqrt{k} >0$ for every $y > \Delta$ for large enough $k$. Therefore,
\begin{align*}
\int_{u=g_k}^\infty  &\Prob\left(S_k \in du | \tau_g > k \right) \tilde{f}_{n-k}(u) \leq \frac{1+o(1)}{\sqrt{2\pi (n-k)}} \left( \Prob\left(\frac{S_k}{\sqrt{k}} \in \left[\frac{g_k}{\sqrt{k}},\frac{g_k}{\sqrt{k}}+\Delta\right) \big| \tau_g > k \right) \right.\\
&\hspace{5cm} \left. +\sum_{z \in \Delta \mathbb{N}} \Prob\left(\frac{S_k}{\sqrt{k}} \in \left[\frac{g_k}{\sqrt{k}} + z,\frac{g_k}{\sqrt{k}} + z+\Delta\right) \big| \tau_g > k \right) e^{-\frac{(g_k + z \sqrt{k})^2}{2(n-k)}} \right).
\end{align*}
Using~\eqref{eq:ConditionedRWDensity}, we obtain
\begin{align*}
\int_{u=g_k}^\infty  &\Prob\left(S_k \in du | \tau_g > k \right) \tilde{f}_{n-k}(u) \leq
\frac{1+o(1)}{\sqrt{2\pi (n-k)}} \left(1-e^{-\frac{\Delta^2}{2}} +  \sum_{z \in \Delta \mathbb{N}} \int_{v=z}^{z+\Delta} v e^{-\frac{v^2}{2}} \, dv \cdot e^{-\frac{z^2}{2}\cdot \frac{k}{n-k}}(1+o(1)) \right)\\
&\leq \frac{1+o(1)}{\sqrt{2\pi (n-k)}} \left(2\left(1-e^{-\frac{\Delta^2}{2}}\right) +  \sum_{z \in \Delta \mathbb{N}_{\geq 2}} \int_{v=z}^{z+\Delta} v e^{-\frac{v^2}{2}}  \cdot e^{-\frac{(v-\Delta)^2 k}{2(n-k)}} \, dv \right)\\
&\leq \frac{1+o(1)}{\sqrt{2\pi (n-k)}} \left(2\left(1-e^{-\frac{\Delta^2}{2}}\right) +   \int_{v=2\Delta}^{\infty} v e^{-\frac{v^2}{2}}  \cdot e^{-\frac{(v-\Delta)^2 k}{2(n-k)}} \, dv \right).
\end{align*}

\noindent
We note that $\lim_{n \rightarrow \infty} k/(n-k) \in [0,\infty)$, and hence
\begin{align*}
\limsup_{n \rightarrow \infty} & \frac{n}{\sqrt{n-k}} \int_{u=g_k }^\infty  \Prob\left(S_k \in du | \tau_g > k \right) \tilde{f}_{n-k}(u) \\
&\leq \frac{1}{\sqrt{2\pi}} \limsup_{n \rightarrow \infty}  \frac{n}{n-k} \lim_{\Delta \downarrow 0} \left(2\left(1-e^{-\frac{\Delta^2}{2}}\right) +   \int_{v=2\Delta}^{\infty} v e^{-\frac{v^2}{2}}  \cdot e^{-\frac{(v-\Delta)^2 }{2} \lim_{n \rightarrow \infty} \frac{k}{n-k}} \, dv \right) =\frac{1}{\sqrt{2\pi}}.
\end{align*}

\noindent
Similarly, we obtain the following lower bound
\begin{align*}
\int_{u=g_k }^\infty  &\Prob\left(S_k \in du | \tau_g > k \right) \tilde{f}_{n-k}(u)  = \int_{v=g_k/\sqrt{k}}^\infty  \Prob\left(\frac{S_k}{\sqrt{k}} \in dv | \tau_g > k \right) \tilde{f}_{n-k}(v \sqrt{k}) \\
&\geq \sum_{z \in \Delta \mathbb{N}_0} \Prob\left(\frac{S_k}{\sqrt{k}} \in \left[\frac{g_k}{\sqrt{k}} + z,\frac{g_k}{\sqrt{k}} + z+\Delta\right) \big| \tau_g > k \right) \inf_{y \in \left[z, z+\Delta\right)} \tilde{f}_{n-k}( g_k + y \sqrt{k}) \\
&\geq \frac{1-o(1)}{\sqrt{2\pi (n-k)}}  \sum_{z \in \Delta \mathbb{N}} \int_{v=z}^{z+\Delta} v e^{-\frac{v^2}{2}} \, dv \cdot e^{-\frac{(z+\Delta)^2}{2}\cdot \frac{k}{n-k}} \\
&\geq \frac{1-o(1)}{\sqrt{2\pi (n-k)}} \int_{v=\Delta}^{\infty} v e^{-\frac{v^2}{2}} \cdot e^{-\frac{(v+2\Delta)^2}{2}\cdot \frac{k}{n-k}} \, dv,
\end{align*}
where we applied~\eqref{eq:RandomWalkConvergenceProperties} and~\eqref{eq:ConditionedRWDensity} again. Therefore,
\begin{align*}
 \liminf_{n \rightarrow \infty} \frac{n}{\sqrt{n-k}} \int_{u=g_k }^\infty & \Prob\left(S_k \in du | \tau_g > k \right) \tilde{f}_{n-k}(u) \\
 &\geq \frac{1}{\sqrt{2\pi}}  \liminf_{n \rightarrow \infty} \frac{n}{n-k} \lim_{\Delta \downarrow 0} \int_{v=\Delta}^{\infty} v e^{-\frac{v^2}{2}} \cdot e^{-\frac{(v+2\Delta)^2}{2}\cdot \lim_{n \rightarrow \infty} \frac{k}{n-k}} \, dv = \frac{1}{\sqrt{2\pi}}.
\end{align*}
The upper and lower bound asymptotically coincide. Using~\eqref{eq:ConditionedRWTauStoppingTime}, we conclude as $n \rightarrow \infty$,
\begin{align*}
\Prob\left(\tau_g > k | S_n = 0  \right) &= \frac{\Prob(\tau_g>k)}{f_n(0)} \int_{u=g_k}^\infty  \Prob\left(S_k \in du | \tau_g > k \right) \tilde{f}_{n-k}(u) \sim \frac{\sqrt{2/\pi} L_g(k) k^{-1/2}}{(2 \pi n)^{-1/2}} \frac{1}{\sqrt{2 \pi}} \frac{\sqrt{n-k}}{n} \\
&= \sqrt{\frac{2}{\pi}} L_g(k) \sqrt{\frac{n-k}{k n}}.
\end{align*}
\end{proof}

\section{Threshold close to return point}
\label{sec:Threshold close to return point}
Unfortunately, the analysis in the previous section does not follow through when $k=n-o(n)$. The chosen grid with small steps of order $\sqrt{k}$ is not refined enough to capture the asymptotic behavior in this case. Another approach is needed, which we elaborate on in this section.

\subsection{Density of random walk at the threshold}
\label{secsec:DensityOfRWAtTheThreshold}
For the evaluation of the objective, it is sensible to consider the position of a random walk at time $k$ itself. A uniform convergence result is given by Proposition~18 of~\cite{Doney2012} in case of constant boundaries. As this result is crucial in our analysis, we pose it here for our setting. 

\begin{proposition}[Proposition~18 of \cite{Doney2012}]
Let $x:=x_n \geq 0$ denote the starting point of a random walk (depending on $n$) and let $y:=y_n$ be a sequence of $n$. Let $U(\cdot)$ denote the renewal function in the (strict) increasing ladder height process, and $V(\cdot)$ the renewal function corresponding to the decreasing ladder height process. Let $\Exp(-S_{T_0})$ be the expected position of a random walk at stopping time $T_0$, and $\Exp(-\tilde{S}_{T_0})$ the expected position of a random walk with increments $-X_i, i\geq 0$ at stopping time $T_0$. Then the following results hold uniformly for every $\Delta \in (0,\infty)$ as $n \rightarrow \infty$.
\begin{itemize}
\item[(i)] For $\max\{x/\sqrt{n},y/\sqrt{n} \} \rightarrow 0$, 
\begin{align}
\Prob\left(S_n \in [y,y+\Delta), T_0 > n | S_0=x \right) \sim \frac{V(x) \int_{y}^{y+\Delta} U(w) \, dw}{\sqrt{2\pi} n^{3/2}}.
\label{eq:Doney1}
\end{align}

\item[(ii)] For any (fixed) $D>1$ with $x/\sqrt{n} \rightarrow 0$ and $y/\sqrt{n} \in [D^{-1},D]$, 
\begin{align}
\Prob\left(S_n \in [y,y+\Delta), T_0 > n | S_0=x \right) \sim \sqrt{\frac{2}{\pi}}\frac{\Exp(-S_{T_0}) V(x)  \Delta}{\sqrt{n}}  \frac{y}{n} e^{-\frac{y^2}{2n}},
\label{eq:Doney2}
\end{align}
and uniformly for $y/\sqrt{n} \rightarrow 0$ and $x/\sqrt{n} \in [D^{-1},D]$,
\begin{align}
\Prob\left(S_n \in [y,y+\Delta), T_0 > n | S_0=x \right) \sim \sqrt{\frac{2}{\pi}} \frac{\Exp(-\tilde{S}_{T_0}) U(y)   \Delta}{\sqrt{n}} \frac{x}{n} e^{-\frac{x^2}{2n}},
\label{eq:Doney3}
\end{align}

\item[(iii)] For any (fixed) $D>1$ with $x/\sqrt{n} \in [D^{-1},D]$ and $y/\sqrt{n} \in [D^{-1},D]$,
\begin{align}
\Prob\left(S_n \in [y,y+\Delta), T_0 > n | S_0=x \right)\sim \frac{\Delta q(x/\sqrt{n},y/\sqrt{n})}{\sqrt{n}},
\label{eq:Doney4}
\end{align}
where $q(x,y)$ is the density of $\Prob(W(1) \in dy, \inf_{0\leq t\leq 1} W(t) >0 | W(0)=x)$ with $\{W(t), t\geq 0\}$ the standard Wiener process. This has the explicit form~\cite{Durrett1977}
\begin{align}
q(u,v) = \frac{1}{\sqrt{2\pi}}\left( e^{\frac{-(u-v)^2}{2}}- e^{\frac{-(u+v)^2}{2}}\right).
\label{eq:DensityBrownianMeander}
\end{align}
for every $u,v > 0$. 
\end{itemize}
\label{prop:DoneyResult}
\end{proposition}

The asymptotic behaviors of $V(\cdot)$ and $U(\cdot)$ are quite well-understood: the functions are both non-decreasing functions and regularly varying with exponent one. In particular, as $t\rightarrow \infty$,
\begin{align}
\begin{array}{ll}
U(t) \sim t/\Exp(-\tilde{S}_{T_0}), & V(t) \sim t/\Exp(-{{S}}_{T_0}).
\end{array}
\label{eq:AssUandV}
\end{align}
Moreover, for all random walks with finite variance $\sigma^2=1$ it holds that
\begin{align}
\Exp(-\tilde{S}_{T_0}) \Exp(-{{S}}_{T_0}) = \frac{\sigma^2}{2} = \frac{1}{2}.
\label{eq:ProductExpectationsLadderHeightsVariance}
\end{align}

The goal is to exploit Proposition~\ref{prop:DoneyResult} to derive the asymptotic behavior of the random walk at time $k$, while staying above the moving boundary. Intuitively, we derive this by looking at the position of the random walk at time $(1-\epsilon)k$, where $\epsilon \in (0,1)$ satisfies~\eqref{eq:AdditionalAssumption}. Due to the additional assumption on $\epsilon$, one can replace the boundary between $(1-\epsilon)k$ and $k$ by the constant boundary $g_k$. The density is then derived using~\eqref{eq:ConditionedRWDensity},~\eqref{eq:ConditionedRWTauStoppingTime} and the result of Doney~\cite{Doney2012} with constant boundaries. This strategy yields the following result.

\begin{proposition}
Let $t \geq g_k$ with $t-g_k=\Omega(|g_k|)$ and $(t-g_k)\rightarrow \infty$ as $k \rightarrow \infty$. Then, uniformly as $k \rightarrow \infty$,
\begin{align*}
\frac{\Prob\left(S_k \in dt ; \tau_g > k \right)}{dt} &\sim \sqrt{\frac{2}{\pi}} \frac{L_g(k)}{k^{3/2}} \cdot \left\{ \begin{array}{ll}
\Exp\left(-\tilde{S}_{T_0} \right) U(t-g_k) & \textrm{if } t=o(\sqrt{k}),\\
t e^{-\frac{t^2}{2k}} & \textrm{if } t = \Omega(\sqrt{k}).\\
\end{array}\right.
\end{align*}
\label{prop:DensityRWConditionedAtTimeK}
\end{proposition}

For the proof of Proposition~\ref{prop:DensityRWConditionedAtTimeK} we use the two following lemmas that show that it is unlikely for the random walk to be close to its boundary at time $(1-\epsilon)k$. 

\begin{lemma}
Suppose $t=o(\sqrt{k})$ such that $t-g_k=\Omega(|g_k|)$ and $(t-g_k)\rightarrow \infty$ as $k \rightarrow \infty$. Let $\epsilon \in (0,1)$ be such that~\eqref{eq:AdditionalAssumption} is satisfied. Choose $x_\epsilon>0$ small enough such that
\begin{align}
1-e^{-\frac{-{x_\epsilon}^2}{2(1-\epsilon)}} <\epsilon^{3/2}
\label{eq:XEpsilonCondition}
\end{align}
holds, and define
\begin{align}
g_{k,\epsilon}^+=\left\{ \begin{array}{ll}
(1+\alpha(\epsilon))g_k & \textrm{if } g_k<0,\\
(1-\alpha(\epsilon))g_k & \textrm{if } g_k\geq 0.
\end{array} \right.
\label{eq:GkepsPlusDefinition}
\end{align}
Then there exist a constant $C_1 \in (0,\infty)$ such that for all $\epsilon \in (0,1)$,
\begin{align*}
\limsup_{k \rightarrow \infty} \frac{k^{3/2}}{L_g(k) U(t-g_k)} \int_{v= g_{(1-\epsilon)k} }^{g_{(1-\epsilon)k} +{x_\epsilon} \sqrt{k}} \Prob\left(S_{(1-\epsilon)k} \in dv ; \tau_g > (1-\epsilon)k \right) &\Prob\left(S_{\epsilon k} \in dt ; T_{g_{k,\epsilon}^+} > \epsilon k | S_0=v \right)  \\
&\leq C_1 \frac{x_\epsilon}{\sqrt{(1-\epsilon)}} \, dt.
\end{align*}
\label{lem:SmallValuesVSmallT}
\end{lemma}

\begin{lemma}
Suppose $t=\Theta(\sqrt{k})$ such that $t\geq g_k$. Let $\epsilon \in (0,1)$ be such that~\eqref{eq:AdditionalAssumption} holds, and choose $x_\epsilon$ small enough such that~\eqref{eq:XEpsilonCondition} is satisfied. Define $g_{k,\epsilon}^+$ as in~\eqref{eq:GkepsPlusDefinition}. Then there exist a constant $C_2 \in (0,\infty)$ such that for all $\epsilon \in (0,1)$,
\begin{align*}
\limsup_{k \rightarrow \infty} \frac{k^{3/2}}{L_g(k) t e^{-t^2/(2k)}} \int_{v= g_{(1-\epsilon)k} }^{g_{(1-\epsilon)k} +{x_\epsilon} \sqrt{k}} \Prob\left(S_{(1-\epsilon)k} \in dv ; \tau_g > (1-\epsilon)k \right) &\Prob\left(S_{\epsilon k} \in dt ; T_{g_{k,\epsilon}^+} > \epsilon k | S_0=v \right)  \\
&\leq C_2 \frac{x_\epsilon}{\sqrt{(1-\epsilon)}} \, dt.
\end{align*}
\label{lem:SmallValuesVLargeT}
\end{lemma}

\noindent
The proofs of Lemma~\ref{lem:SmallValuesVSmallT} and~\ref{lem:SmallValuesVLargeT} are given in Appendix~\ref{app:ProofOfLemmas}. Next, we will prove Proposition~\ref{prop:DensityRWConditionedAtTimeK}.

\begin{proof}[Proof of Proposition~\ref{prop:DensityRWConditionedAtTimeK}]
We consider the position of the random walk at time $(1-\epsilon)k$ with $\epsilon \in (0,1)$. Specifically, fix $\epsilon>0$ such that~\eqref{eq:AdditionalAssumption} holds and, additionally,
\begin{align}
\alpha(\epsilon) < \lim_{k \rightarrow \infty} \frac{t-g_k}{|g_k|}
\label{eq:AddExtraAssumptionOnEpsilon}
\end{align}
is satisfied. Note that
\begin{align*}
\Prob&\left(S_k \in dt ; \tau_g > k \right) = \int_{v=g_{(1-\epsilon)k}}^\infty \Prob\left(S_k \in dt; \tau_g > k ; S_{(1-\epsilon)k} \in dv \right)\\
&\hspace{1cm}= \int_{v=g_{(1-\epsilon)k}}^\infty \Prob\left(S_k \in dt ; \tau_g > k | S_{(1-\epsilon)k} = v ; \tau_g> (1-\epsilon)k \right)\Prob\left( S_{(1-\epsilon)k} \in dv ; \tau_g > (1-\epsilon) k \right) \\
&\hspace{1cm}=\int_{v=g_{(1-\epsilon)k}}^\infty \Prob\left(S_{\epsilon k} \in dt ; \tau_{\breve{g}^\epsilon} > \epsilon k | S_0=v \right)\Prob\left( S_{(1-\epsilon)k} \in dv | \tau_g > (1-\epsilon) k \right) \Prob\left( \tau_g > (1-\epsilon) k \right) ,
\end{align*}
where $\breve{g}^\epsilon, i\geq 1$ is defined as
\begin{align*}
\breve{g}^\epsilon_j := g_{(1-\epsilon)k+j}, \hspace{1cm} \forall 1 \leq j \leq \epsilon k.
\end{align*}
Observe that~\eqref{eq:AdditionalAssumption} is then equivalent to
\begin{align*}
\sup_{1 \leq j \leq k} \left|\frac{\breve{g}^\epsilon_j}{g_k}-1\right| \leq \alpha(\epsilon).
\end{align*}

\noindent
Define $g_{k,\epsilon}^+$ as in~\eqref{eq:GkepsPlusDefinition} and let 
\begin{align}
g_{k,\epsilon}^-=\left\{ \begin{array}{ll}
(1-\alpha(\epsilon))g_k & \textrm{if } g_k<0,\\
(1+\alpha(\epsilon))g_k & \textrm{if } g_k\geq 0.
\end{array} \right.
\label{eq:GkepsMinusDefinition}
\end{align}
We obtain the bounds
\begin{align}
\Prob\left(S_k \in dt ; \tau_g > k \right) \leq \int_{v=g_{(1-\epsilon)k}}^\infty \Prob\left(S_{(1-\epsilon)k} \in dv ; \tau_g > (1-\epsilon)k \right) \Prob\left(S_{\epsilon k} \in dt ; T_{g_{k,\epsilon}^+} > \epsilon k | S_0=v \right),
\label{eq:UBDensityAtK}
\end{align}
and
\begin{align}
\Prob\left(S_k \in dt ; \tau_g > k \right) \geq \int_{v=g_{(1-\epsilon)k}}^\infty \Prob\left(S_{(1-\epsilon)k} \in dv ; \tau_g > (1-\epsilon)k \right) \Prob\left(S_{\epsilon k} \in dt ; T_{g_{k,\epsilon}^-} > \epsilon k | S_0=v \right).
\label{eq:LBDensityAtK}
\end{align}
The goal is to show that the bounds asymptotically coincide when $\epsilon \downarrow 0$.

Let $v_{\epsilon,k}=g_{(1-\epsilon)k}+x_\epsilon \sqrt{k}$, where $x_\epsilon>0$ is chosen small enough such that
\begin{align}
1-e^{-\frac{x_\epsilon^2}{2(1-\epsilon)}} < \epsilon^{3/2}.
\label{eq:HelpForSmallV}
\end{align}
Note that this choice of $x_\epsilon$ implies that $x_\epsilon/\sqrt{\epsilon(1-\epsilon)} \rightarrow 0$ as $\epsilon \downarrow 0$, since
\begin{align*}
\lim_{\epsilon \downarrow 0} \frac{x_\epsilon}{\sqrt{\epsilon(1-\epsilon)}} < \lim_{\epsilon \downarrow 0} \sqrt{\frac{-2\log(1-\epsilon^{3/2})}{\epsilon}} =0.
\end{align*}
Then, \eqref{eq:UBDensityAtK} can be written as
\begin{align*}
\Prob\left(S_k \in dt ; \tau_g > k \right) \leq \int_{v=g_{(1-\epsilon)k}}^{v_{\epsilon,k}} &\Prob\left(S_{(1-\epsilon)k} \in dv ; \tau_g > (1-\epsilon)k \right) \Prob\left(S_{\epsilon k} \in dt ; T_{g_{k,\epsilon}^+} > \epsilon k | S_0=v \right) \\
&+ \int_{v=v_{\epsilon,k}}^\infty \Prob\left(S_{(1-\epsilon)k} \in dv ; \tau_g > (1-\epsilon)k \right) \Prob\left(S_{\epsilon k} \in dt ; T_{g_{k,\epsilon}^+} > \epsilon k | S_0=v \right).
\end{align*}

\noindent
Lemma~\ref{lem:SmallValuesVSmallT} and~\ref{lem:SmallValuesVLargeT} provide a bound for the first term, i.e.~there exists $c_1,c_2 \in (0,\infty)$ such that for every $\epsilon \in (0,1)$ satisfying~\eqref{eq:AdditionalAssumption},
\begin{align}
\begin{split}
\int_{v= g_{(1-\epsilon)k} }^{v_{\epsilon,k}} &\Prob\left(S_{(1-\epsilon)k} \in dv ; \tau_g > (1-\epsilon)k \right) \Prob\left(S_{\epsilon k} \in dt ; T_{g_{k,\epsilon}^+} > \epsilon k | S_0=v \right) \\
&\hspace{4cm} \leq  \frac{L_g(k)}{k^{3/2}} \frac{x_\epsilon}{\sqrt{1-\epsilon}} \left\{\begin{array}{ll}
c_1 U(t-g_k) \, dt & \textrm{if } |t|=o(\sqrt{k}), \\
c_2 t e^{-\frac{t^2}{2k}} \, dt & \textrm{if } |t|=\Theta(\sqrt{k}). \\
\end{array}\right.
\end{split}
\label{eq:BeginPart}
\end{align}
To evaluate the second integral term, we use a similar analysis as we have done for the case where $k$ is sufficiently far from $n$. First, due to Proposition~\ref{prop:DoneyResult}, it holds uniformly as $k \rightarrow \infty$,
\begin{align}
\begin{split}
\frac{\Prob\left(S_{\epsilon k} \in dt ; T_{g_{k,\epsilon}^+} > \epsilon k | S_0=v \right)}{dt} &= \frac{\Prob\left(S_{\epsilon k} \in d(t-g_{k,\epsilon}^+) ; T_{0} > \epsilon k | S_0=v-g_{k,\epsilon}^+ \right)}{d(t-g_{k,\epsilon}^+)}\\ 
&\sim \left\{ \begin{array}{ll}
\sqrt{\frac{2}{\pi}} \Exp(-\tilde{S}_{T_0}) \frac{U(t-g_{k,\epsilon}^+)}{\sqrt{\epsilon k}} \frac{v-g_{k,\epsilon}^+}{\epsilon k} e^{-\frac{(v-g_{k,\epsilon}^+)^2}{2\epsilon k}}   &\textrm{if } t=o(\sqrt{k}),\\
\frac{1}{\sqrt{2 \pi \epsilon k }} \left( e^{-\frac{(v-t)^2}{2\epsilon k}} - e^{-\frac{(v+t-2g_{k,\epsilon}^+)^2}{2\epsilon k}} \right)  &\textrm{if } t=\Theta(\sqrt{k}).
\end{array}\right.
\end{split}
\label{eq:AssLastPart}
\end{align}
Moreover, using~\eqref{eq:ConditionedRWDensity},~\eqref{eq:ConditionedRWTauStoppingTime} and the fact that $L_g(\cdot)$ is slowly varying,
\begin{align}
\begin{split}
\Prob\left(S_{(1-\epsilon)k} \in dv ; \tau_g > (1-\epsilon)k \right) &= \Prob\left(S_{(1-\epsilon)k} \in dv | \tau_g > (1-\epsilon)k \right)\Prob\left(\tau_g > (1-\epsilon)k \right)\\
&\sim \sqrt{\frac{2}{\pi}} \frac{L_g(k)}{\sqrt{(1-\epsilon)k}} \frac{v-g_{(1-\epsilon)k}}{(1-\epsilon)k} e^{-\frac{(v-g_{(1-\epsilon)k})^2}{2(1-\epsilon)k}} \, dv.
\end{split}
\label{eq:AssFirstPart}
\end{align}

\noindent
These asymptotic results can be used to evaluate the second term. For readability of the proof, we complete the proof separately for the two cases stated in the proposition, i.e.~for $t=o(\sqrt{k})$ and $t=\Theta(\sqrt{k})$. Fix $\Delta<x_\epsilon/2$ and recall
\begin{align*}
\Delta \mathbb{N}_0 = \{ \Delta i : i \geq 0, i\in \mathbb{Z}\}.
\end{align*}
In the case $t=o(\sqrt{k})$, using~\eqref{eq:AssFirstPart} and~\eqref{eq:AssLastPart}, we find for every fixed $0 < \Delta < x_\epsilon/2$,
\begin{align*}
\limsup_{k \rightarrow \infty} &\frac{k^{3/2}}{L_g(k) U(t-g_{k, \epsilon}^+)}  \int_{v=v_{\epsilon,k}}^\infty \Prob\left(S_{(1-\epsilon)k} \in dv ; \tau_g > (1-\epsilon)k \right) \frac{\Prob\left(S_{\epsilon k} \in dt ; T_{g_{k,\epsilon}^+} > \epsilon k | S_0=v \right)}{dt}\\
&\leq \limsup_{k \rightarrow \infty} \frac{k^{3/2}}{L_g(k) U(t-g_{k, \epsilon}^+)}  \left(\sum_{z \in \Delta \mathbb{N}_0} \Prob\left(S_{(1-\epsilon)k} \in [v_{\epsilon,k} +z\sqrt{k},v_{\epsilon,k} +(z+\Delta)\sqrt{k}) ; \tau_g > (1-\epsilon)k \right) \right.\\
&\hspace{5cm} \left.\cdot \sup_{v \in [v_{\epsilon,k} +z\sqrt{k},v_{\epsilon,k} +(z+\Delta)\sqrt{k})} \frac{\Prob\left(S_{\epsilon k} \in dt ; T_{g_{k,\epsilon}^+} > \epsilon k | S_0=v \right)}{dt} \right)\\
&\leq \limsup_{k \rightarrow \infty} \frac{k^{3/2}}{L_g(k) U(t-g_{k, \epsilon}^+)} \left( \sum_{z \in \Delta \mathbb{N}_0} \int_{v=v_{\epsilon,k}+z\sqrt{k}}^{v_{\epsilon,k}+(z+\Delta)\sqrt{k}} \sqrt{\frac{2}{\pi}} \frac{L_g(k)}{\sqrt{(1-\epsilon)k}} \frac{v-g_{(1-\epsilon)k}}{(1-\epsilon)k} e^{-\frac{(v-g_{(1-\epsilon)k})^2}{2(1-\epsilon)k}} \, dv \right.\\
&\hspace{3cm} \left.\cdot \sqrt{\frac{2}{\pi}} \Exp(-\tilde{S}_{T_0}) \frac{U(t-g_{k,\epsilon}^+)}{\sqrt{\epsilon k}} \frac{v_{\epsilon,k}+(z+\Delta)\sqrt{k}-g_{k,\epsilon}^+}{\epsilon k} e^{-\frac{(v_{\epsilon,k}+z\sqrt{k}-g_{k,\epsilon}^+)^2}{2\epsilon k}} \right)\\
&\leq \frac{2}{\pi} \frac{\Exp(-\tilde{S}_{T_0})}{\sqrt{\epsilon(1-\epsilon)}}  \int_{y=x_\epsilon}^\infty \frac{y}{(1-\epsilon)} e^{-\frac{y^2}{2(1-\epsilon)}} \frac{y+2\Delta}{\epsilon } e^{-\frac{(y-2\Delta)^2}{2\epsilon }} \, dy,
\end{align*}
where the final inequality is due to variable substitution, and since for every $v \in [v_{\epsilon,k}+z\sqrt{k},v_{\epsilon,k}+(z+\Delta)\sqrt{k}]$,
\begin{align*}
v_{\epsilon,k}+(z+\Delta)\sqrt{k}-g_{k,\epsilon}^+ \leq v-g_{(1-\epsilon)k}+2\Delta\sqrt{k}
\end{align*}
and 
\begin{align*}
v_{\epsilon,k}+z\sqrt{k}-g_{k,\epsilon}^+ \geq v-g_{(1-\epsilon)k}-2\Delta\sqrt{k}
\end{align*}
for $k$ large enough. Letting $\Delta \downarrow 0$ and invoking Lemma~\ref{lem:Integral1VSmall} yields
\begin{align*}
\limsup_{k \rightarrow \infty} \frac{k^{3/2}}{L_g(k) U(t-g_{k, \epsilon}^+)} & \int_{v=v_{\epsilon,k}}^\infty \Prob\left(S_{(1-\epsilon)k} \in dv ; \tau_g > (1-\epsilon)k \right) \frac{\Prob\left(S_{\epsilon k} \in dt ; T_{g_{k,\epsilon}^+} > \epsilon k | S_0=v \right)}{dt}\\
&\leq \lim_{\Delta \downarrow 0} \left(\frac{2}{\pi} \Exp(-\tilde{S}_{T_0}) \frac{1}{\sqrt{\epsilon(1-\epsilon)}} \int_{y=0}^\infty \frac{y(y+2\Delta)}{\epsilon(1-\epsilon)} e^{-\frac{y^2}{2(1-\epsilon)}-\frac{(y-2\Delta)^2}{2\epsilon }} \, dy \right) = \sqrt{\frac{2}{\pi}} \Exp(-\tilde{S}_{T_0}) .
\end{align*}
Recall~\eqref{eq:GkepsPlusDefinition} and~\eqref{eq:AssUandV}. Hence,
\begin{align*}
\limsup_{k \rightarrow \infty} \frac{U(t-g_{k,\epsilon}^+) }{U(t-g_k)} \leq 1+\alpha(\epsilon) \limsup_{k \rightarrow \infty} \frac{|g_k|}{t-g_k} < \infty.
\end{align*}
Combining this with~\eqref{eq:BeginPart}, we obtain the upper bound
\begin{align*}
\limsup_{k \rightarrow \infty}  \frac{k^{3/2}}{L_g(k)U(t-g_k)}& \frac{\Prob\left(S_k \in dt ; \tau_g > k \right)}{dt} \\
&\leq \lim_{\epsilon \downarrow 0} \left(c_1 \frac{x_\epsilon}{\sqrt{1-\epsilon}} + \sqrt{\frac{2}{\pi}} \Exp(-\tilde{S}_{T_0})\left(1+\alpha(\epsilon) \limsup_{k \rightarrow \infty} \frac{|g_k|}{t-g_k}\right) \right) = \sqrt{\frac{2}{\pi}} \Exp(-\tilde{S}_{T_0}).
\end{align*}

\noindent
Similarly, using~\eqref{eq:AssLastPart} and~\eqref{eq:AssFirstPart}, we obtain a lower bound for every fixed $0 < \Delta < x_\epsilon/2$,
\begin{align*}
\liminf_{k \rightarrow \infty} &\frac{k^{3/2}}{L_g(k) U(t-g_{k, \epsilon}^-)}  \int_{v=v_{\epsilon,k}}^\infty \Prob\left(S_{(1-\epsilon)k} \in dv ; \tau_g > (1-\epsilon)k \right) \frac{\Prob\left(S_{\epsilon k} \in dt ; T_{g_{k,\epsilon}^-} > \epsilon k | S_0=v \right)}{dt}\\
&\geq \liminf_{k \rightarrow \infty} \left(\frac{k^{3/2}}{L_g(k) U(t-g_{k, \epsilon}^-)} \sum_{z \in \Delta \mathbb{N}_0} \Prob\left(S_{(1-\epsilon)k} \in [v_{\epsilon,k} +z\sqrt{k},v_{\epsilon,k} +(z+\Delta)\sqrt{k}) ; \tau_g > (1-\epsilon)k \right) \right.\\
&\hspace{5cm} \left. \cdot \inf_{v \in [v_{\epsilon,k} +z\sqrt{k},v_{\epsilon,k} +(z+\Delta)\sqrt{k})} \frac{\Prob\left(S_{\epsilon k} \in dt ; T_{g_{k,\epsilon}^-} > \epsilon k | S_0=v \right)}{dt} \right)\\
&\geq \liminf_{k \rightarrow \infty} \left(\frac{k^{3/2}}{L_g(k) U(t-g_{k, \epsilon}^-)} \sum_{z \in \Delta \mathbb{N}_0} \int_{v=v_{\epsilon,k}+z\sqrt{k}}^{v_{\epsilon,k}+(z+\Delta)\sqrt{k}} \sqrt{\frac{2}{\pi}} \frac{L_g(k)}{\sqrt{(1-\epsilon)k}} \frac{v-g_{(1-\epsilon)k}}{(1-\epsilon)k} e^{-\frac{(v-g_{(1-\epsilon)k})^2}{2(1-\epsilon)k}} \, dv \right.\\
&\hspace{5cm} \left. \cdot \sqrt{\frac{2}{\pi}} \Exp(-\tilde{S}_{T_0}) \frac{U(t-g_{k,\epsilon}^-)}{\sqrt{\epsilon k}} \frac{v_{\epsilon,k}+z\sqrt{k}-g_{k,\epsilon}^-}{\epsilon k} e^{-\frac{(v_{\epsilon,k}+(z+\Delta)\sqrt{k}-g_{k,\epsilon}^-)^2}{2\epsilon k}} \right)\\
&\geq \frac{2}{\pi}  \Exp(-\tilde{S}_{T_0})\frac{1}{\sqrt{(1-\epsilon)\epsilon} } \int_{y=x_\epsilon}^\infty \frac{y}{(1-\epsilon)} e^{-\frac{y^2}{2(1-\epsilon)}} \frac{y-2\Delta}{\epsilon } e^{-\frac{(y+2\Delta)^2}{2\epsilon k}} \, dy.
\end{align*}
Invoking Lemma~\ref{lem:Integral1VSmall} yields
\begin{align*}
\liminf_{k \rightarrow \infty} &\frac{k^{3/2}}{L_g(k) U(t-g_{k, \epsilon}^-)}  \int_{v=v_{\epsilon,k}}^\infty \Prob\left(S_{(1-\epsilon)k} \in dv ; \tau_g > (1-\epsilon)k \right) \frac{\Prob\left(S_{\epsilon k} \in dt ; T_{g_{k,\epsilon}^-} > \epsilon k | S_0=v \right)}{dt}\\
&\geq \lim_{\Delta \downarrow 0} \left( \frac{2}{\pi}  \Exp(-\tilde{S}_{T_0})\frac{1}{\sqrt{(1-\epsilon)\epsilon} } \int_{y=x_\epsilon}^\infty \frac{y}{(1-\epsilon)} e^{-\frac{y^2}{2(1-\epsilon)}} \frac{y-2\Delta}{\epsilon } e^{-\frac{(y+2\Delta)^2}{2\epsilon k}} \, dy \right)\\
&= \frac{2}{\pi}  \Exp(-\tilde{S}_{T_0}) \left( \frac{x_\epsilon}{\sqrt{\epsilon(1-\epsilon)}} e^{-\frac{x_\epsilon}{2\epsilon(1-\epsilon)}} + \int_{y=x_\epsilon/\sqrt{\epsilon(1-\epsilon)}}^\infty e^{-y^2/2} \, dy \right).
\end{align*}
Recall that the choice of $x_\epsilon$ ensures that $x_\epsilon/\sqrt{\epsilon(1-\epsilon)} \rightarrow 0$ as $\epsilon \downarrow 0$. Moreover, we observe that due to \eqref{eq:AddExtraAssumptionOnEpsilon}, \eqref{eq:GkepsMinusDefinition} and~\eqref{eq:AssUandV},
\begin{align*}
\liminf_{k \rightarrow \infty} \frac{U(t-g_{k,\epsilon}^-) }{U(t-g_k)} \geq 1-\alpha(\epsilon) \liminf_{k \rightarrow \infty} \frac{|g_k|}{t-g_k}.
\end{align*}
Therefore,
\begin{align*}
& \liminf_{k \rightarrow \infty} \frac{k^{3/2}}{L_g(k)U(t-g_k)} \frac{\Prob\left(S_k \in dt ; \tau_g > k \right)}{dt} \\
&\;\;\; \geq \frac{2}{\pi}\Exp(-\tilde{S}_{T_0}) \lim_{\epsilon \downarrow 0} \left(1-\alpha(\epsilon) \liminf_{k \rightarrow \infty} \frac{|g_k|}{t-g_k} \right) \left( \frac{x_\epsilon}{\sqrt{\epsilon(1-\epsilon)}} e^{-\frac{x_\epsilon}{2\epsilon(1-\epsilon)}}+ \int_{y=\frac{x_\epsilon}{\sqrt{\epsilon(1-\epsilon)}}}^\infty e^{-y^2/2} \, dy\right)= \sqrt{\frac{2}{\pi}}\Exp(-\tilde{S}_{T_0}).
\end{align*}
Note this coincides with the upper bound, proving the result in case of $t=o(\sqrt{k})$.

Next, we complete the proof in case of $t=\Theta(\sqrt{k})$. Again, using~\eqref{eq:AssLastPart} and~\eqref{eq:AssFirstPart}, we obtain for any $0<\Delta <\min\{x_\epsilon,\lim_{k \rightarrow \infty} t/\sqrt{k}\}/3$,
\begin{align*}
\limsup_{k \rightarrow \infty}  & \frac{k^{3/2}}{L_g(k)t e^{-t^2/(2k)}}  \int_{v=v_{\epsilon,k}}^\infty \Prob\left(S_{(1-\epsilon)k} \in dv ; \tau_g > (1-\epsilon)k \right) \frac{\Prob\left(S_{\epsilon k} \in dt ; T_{g_{k,\epsilon}^+} > \epsilon k | S_0=v \right)}{dt}\\
&\leq \limsup_{k \rightarrow \infty}  \frac{k^{3/2}}{L_g(k)t e^{-t^2/(2k)}}  \sum_{z \in \Delta \mathbb{N}_0} \Prob\left(S_{(1-\epsilon)k} \in [v_{\epsilon,k} +z\sqrt{k},v_{\epsilon,k} +(z+\Delta)\sqrt{k}) ; \tau_g > (1-\epsilon)k \right)\\
&\hspace{5cm} \cdot \sup_{v \in [v_{\epsilon,k} +z\sqrt{k},v_{\epsilon,k} +(z+\Delta)\sqrt{k})} \frac{\Prob\left(S_{\epsilon k} \in dt ; T_{g_{k,\epsilon}^+} > \epsilon k | S_0=v \right)}{dt}\\
&\leq \limsup_{k \rightarrow \infty}  \frac{k^{3/2}}{L_g(k)t e^{-t^2/(2k)}}  \sum_{z \in \Delta \mathbb{N}_0} \int_{v=v_{\epsilon,k}+z\sqrt{k}}^{v_{\epsilon,k}+(z+\Delta)\sqrt{k}} \sqrt{\frac{2}{\pi}} \frac{L_g(k)}{\sqrt{(1-\epsilon)k}} \frac{v-g_{(1-\epsilon)k}}{(1-\epsilon)k} e^{-\frac{(v-g_{(1-\epsilon)k})^2}{2(1-\epsilon)k}} \, dv\\
&\hspace{5cm} \cdot \sup_{v \in [v_{\epsilon,k} +z\sqrt{k},v_{\epsilon,k} +(z+\Delta)\sqrt{k})} \frac{1}{\sqrt{2 \pi \epsilon k }} \left( e^{-\frac{(v-t)^2}{2\epsilon k}} - e^{-\frac{(v+t-2g_{k,\epsilon}^+)^2}{2\epsilon k}} \right).
\end{align*}
Since for any $a\leq b$,
\begin{align*}
\sup_{v \in [a,b]} e^{-\frac{(v-t)^2}{2\epsilon k}} =\left\{ \begin{array}{ll}
e^{-\frac{(b-t)^2}{2\epsilon k}} & \textrm{if } b \leq t,\\
e^{-\frac{(a-t)^2}{2\epsilon k}} & \textrm{if } a \geq t,\\
1 & \textrm{othterwise,} \\
 \end{array}\right.
\end{align*}
we obtain 
\begin{align*}
\limsup_{k \rightarrow \infty}  & \frac{k^{3/2}}{L_g(k)t e^{-t^2/(2k)}}  \int_{v=v_{\epsilon,k}}^\infty \Prob\left(S_{(1-\epsilon)k} \in dv ; \tau_g > (1-\epsilon)k \right) \frac{\Prob\left(S_{\epsilon k} \in dt ; T_{g_{k,\epsilon}^+} > \epsilon k | S_0=v \right)}{dt}\\
&\leq \limsup_{k \rightarrow \infty} \frac{1}{\pi}   \frac{k^{1/2}}{t e^{-t^2/(2k)} \sqrt{\epsilon(1-\epsilon)}}  \left( \int_{y=x_\epsilon}^{t/\sqrt{k}-3\Delta} \frac{y}{(1-\epsilon)} e^{-\frac{y^2}{2(1-\epsilon)}} \left( e^{-\frac{(y+3\Delta-t/\sqrt{k})^2}{2\epsilon }} - e^{-\frac{(y+3\Delta+t/\sqrt{k})^2}{2\epsilon }} \right)\, dy \right.\\
&\hspace{4.25cm} +\left. \int_{y=t/\sqrt{k}+3\Delta}^\infty \frac{y}{(1-\epsilon)} e^{-\frac{y^2}{2(1-\epsilon)}} \left( e^{-\frac{(y-3\Delta-t/\sqrt{k})^2}{2\epsilon }} - e^{-\frac{(y+3\Delta+t/\sqrt{k})^2}{2\epsilon }} \right)\, dy \right.\\
&\hspace{5cm} +\left. \int_{y=t/\sqrt{k}-3\Delta}^{t/\sqrt{k}+3\Delta} \frac{y}{(1-\epsilon)} e^{-\frac{y^2}{2(1-\epsilon)}} \left( 1 - e^{-\frac{(y+3\Delta+t/\sqrt{k})^2}{2\epsilon}} \right)\, dy \right).
\end{align*}
Since $t=\Theta(\sqrt{k})$,
\begin{align*}
\limsup_{k \rightarrow \infty}  & \frac{k^{3/2}}{L_g(k)t e^{-t^2/(2k)}}  \int_{v=v_{\epsilon,k}}^\infty \Prob\left(S_{(1-\epsilon)k} \in dv ; \tau_g > (1-\epsilon)k \right) \frac{\Prob\left(S_{\epsilon k} \in dt ; T_{g_{k,\epsilon}^+} > \epsilon k | S_0=v \right)}{dt}\\
&\leq \lim_{\Delta \downarrow 0} \limsup_{k \rightarrow \infty} \frac{k^{3/2}}{L_g(k)t e^{-t^2/(2k)}}  \left( \int_{y=x_\epsilon}^{t/\sqrt{k}-3\Delta} \frac{y}{(1-\epsilon)} e^{-\frac{y^2}{2(1-\epsilon)}} \left( e^{-\frac{(y+3\Delta-t/\sqrt{k})^2}{2\epsilon }} - e^{-\frac{(y+3\Delta+t/\sqrt{k})^2}{2\epsilon }} \right)\, dy \right.\\
&\hspace{4.75cm} +\left. \int_{y=t/\sqrt{k}+3\Delta}^\infty \frac{y}{(1-\epsilon)} e^{-\frac{y^2}{2(1-\epsilon)}} \left( e^{-\frac{(y-3\Delta-t/\sqrt{k})^2}{2\epsilon }} - e^{-\frac{(y+3\Delta+t/\sqrt{k})^2}{2\epsilon }} \right)\, dy \right.\\
&\hspace{5.5cm} +\left. \int_{y=t/\sqrt{k}-3\Delta}^{t/\sqrt{k}+3\Delta} \frac{y}{(1-\epsilon)} e^{-\frac{y^2}{2(1-\epsilon)}} \left( 1 - e^{-\frac{(y+3\Delta+t/\sqrt{k})^2}{2\epsilon}} \right)\, dy \right)\\
&= \lim_{k \rightarrow \infty} \frac{1}{\pi}   \frac{k^{1/2}}{t e^{-t^2/(2k)} \sqrt{\epsilon(1-\epsilon)}} \cdot \int_{y=x_\epsilon}^\infty \frac{y}{(1-\epsilon)} e^{-\frac{y^2}{2(1-\epsilon)}} \left( e^{-\frac{(y-\lim_{k \rightarrow \infty} t/\sqrt{k})^2}{2\epsilon }} - e^{-\frac{(y+\lim_{k \rightarrow \infty} t/\sqrt{k})^2}{2\epsilon }} \right)  \, dy
\end{align*}
with $x_\epsilon>0$. Invoking Lemma~\ref{lem:Integral2VLarge} concludes
\begin{align*}
\limsup_{k \rightarrow \infty}  & \frac{k^{3/2}}{L_g(k)t e^{-t^2/(2k)}}  \int_{v=v_{\epsilon,k}}^\infty \Prob\left(S_{(1-\epsilon)k} \in dv ; \tau_g > (1-\epsilon)k \right) \frac{\Prob\left(S_{\epsilon k} \in dt ; T_{g_{k,\epsilon}^+} > \epsilon k | S_0=v \right)}{dt}\\
&\leq\limsup_{k \rightarrow \infty} \frac{1}{\pi}   \frac{k^{1/2}}{t e^{-t^2/(2k)} \sqrt{\epsilon(1-\epsilon)}}  \int_{y=0}^\infty \frac{y}{(1-\epsilon)} e^{-\frac{y^2}{2(1-\epsilon)}} \left( e^{-\frac{(y-\lim_{k \rightarrow \infty} t/\sqrt{k})^2}{2\epsilon }} - e^{-\frac{(y+\lim_{k \rightarrow \infty} t/\sqrt{k})^2}{2\epsilon }} \right)  \, dy\\
&=\frac{1}{\pi} \cdot \sqrt{2\pi}=\sqrt{\frac{2}{\pi}}.
\end{align*}
Combining this expression with~\eqref{eq:BeginPart} and letting $\epsilon \downarrow 0$, we obtain the upper bound
\begin{align*}
\limsup_{k \rightarrow \infty}  & \frac{k^{3/2}}{L_g(k)t e^{-t^2/(2k)}} \frac{\Prob(S_k \in dt ; \tau_g>k)}{dt} \leq \sqrt{\frac{2}{\pi}}.
\end{align*}

\noindent
For the lower bound, using~\eqref{eq:AssLastPart} and~\eqref{eq:AssFirstPart}, we obtain for any $0<\Delta <\min\{x_\epsilon,\lim_{k \rightarrow \infty} t/\sqrt{k}\}/3$,
\begin{align*}
\liminf_{k \rightarrow \infty}  & \frac{k^{3/2}}{L_g(k)t e^{-t^2/(2k)}} \int_{v=v_{\epsilon,k}}^\infty \Prob\left(S_{(1-\epsilon)k} \in dv ; \tau_g > (1-\epsilon)k \right) \frac{\Prob\left(S_{\epsilon k} \in dt ; T_{g_{k,\epsilon}^-} > \epsilon k | S_0=v \right)}{dt}\\
&\geq \liminf_{k \rightarrow \infty}  \frac{k^{3/2}}{L_g(k)t e^{-t^2/(2k)}} \sum_{z \in \Delta \mathbb{N}_0} \Prob\left(S_{(1-\epsilon)k} \in [v_{\epsilon,k} +z\sqrt{k},v_{\epsilon,k} +(z+\Delta)\sqrt{k}) ; \tau_g > (1-\epsilon)k \right)\\
&\hspace{5cm} \cdot \inf_{v \in [v_{\epsilon,k} +z\sqrt{k},v_{\epsilon,k} +(z+\Delta)\sqrt{k})} \frac{\Prob\left(S_{\epsilon k} \in dt ; T_{g_{k,\epsilon}^-} > \epsilon k | S_0=v \right)}{dt}\\
&\geq \liminf_{k \rightarrow \infty}  \frac{1}{\pi} \frac{k^{1/2}}{t e^{-t^2/(2k)}\sqrt{\epsilon(1-\epsilon)}}  \sum_{z \in \Delta \mathbb{N}_0} \int_{v=v_{\epsilon,k}+z\sqrt{k}}^{v_{\epsilon,k}+(z+\Delta)\sqrt{k}}  \frac{v-g_{(1-\epsilon)k}}{(1-\epsilon)k} e^{-\frac{(v-g_{(1-\epsilon)k})^2}{2(1-\epsilon)k}} \, dv\\
&\hspace{5cm} \cdot \inf_{v \in [v_{\epsilon,k} +z\sqrt{k},v_{\epsilon,k} +(z+\Delta)\sqrt{k})} \left( e^{-\frac{(v-t)^2}{2\epsilon k}} - e^{-\frac{(v+t-2g_{k,\epsilon}^-)^2}{2\epsilon k}} \right) \\
&\geq \liminf_{k \rightarrow \infty}  \frac{1}{\pi} \frac{k^{1/2}}{t e^{-t^2/(2k)}\sqrt{\epsilon(1-\epsilon)}}  \left( \int_{y=x_\epsilon}^{t/\sqrt{k}-3\Delta} \frac{y}{(1-\epsilon)} e^{-\frac{y^2}{2(1-\epsilon)}} \left( e^{-\frac{(y-3\Delta- t/\sqrt{k})^2}{2\epsilon }} - e^{-\frac{(y-3\Delta+ t/\sqrt{k})^2}{2\epsilon }} \right)\, dy \right.\\
&\hspace{4.5cm} +\left. \int_{y=t/\sqrt{k}+3\Delta}^\infty \frac{y}{(1-\epsilon)} e^{-\frac{y^2}{2(1-\epsilon)}} \left( e^{-\frac{(y+3\Delta-t/\sqrt{k})^2}{2\epsilon}} - e^{-\frac{(y-3\Delta+t/\sqrt{k})^2}{2\epsilon }} \right)\, dy \right).
\end{align*}
We observe that the limiting expression tends to a constant as $k \rightarrow \infty$. Then, setting $\Delta \downarrow 0$ and invoking Lemma~\ref{lem:Integral2VLarge} yields
\begin{align*}
\liminf_{k\rightarrow\infty} \frac{k^{3/2}}{L_g(k)t e^{-t^2/(2k)}} \frac{\Prob\left(S_k \in dt ; \tau_g > k \right)}{dt} \geq \frac{1}{\pi}\left(\sqrt{2\pi} - \frac{c_3 \pi}{\sqrt{\epsilon(1-\epsilon)}}\left( 1-e^{-\frac{x_\epsilon^2}{2(1-\epsilon)}}\right) \right) \geq \sqrt{\frac{2}{\pi}} -c_3 \frac{\epsilon}{\sqrt{1-\epsilon}}
\end{align*}
for some constant $c_3>0$. Finally, we observe that as $\epsilon \downarrow 0$ the upper and lower bound coincode, concluding the proof for $t=\Theta(\sqrt{k})$.
\end{proof}

\subsection{Proof of Theorem~\ref{thm:MainResultBoundary}}
Recall that $\tilde{S}_m$, $m\geq 1$ denotes the reversed random walk defined in~\eqref{eq:DefinitionReversedRW}, and $\tilde{f}_m(\cdot)$ the corresponding density function at time $m$. Then,
\begin{align*}
\Prob\left(\tau_g > k | S_n = 0  \right) &= \int_{u=g_k}^\infty \Prob\left(\tau_g > k ; S_k \in du | S_n = 0  \right) = \frac{1}{f_n(0)} \int_{u=g_k}^\infty  \Prob\left(S_k \in du ; \tau_g > k \right) \tilde{f}_{n-k}(u).
\end{align*}

\noindent
Since the reversed random walk also satisfies property~\eqref{eq:RandomWalkConvergenceProperties}, 
\begin{align*}
\Prob\left(\tau_g > k | S_n = 0  \right) &\sim \frac{1}{f_n(0)\sqrt{2\pi(n-k)}} \int_{u=g_k}^\infty  \Prob\left(S_k \in du ; \tau_g > k \right) e^{-\frac{u^2}{2(n-k)}} \\
&\sim \sqrt{\frac{n}{n-k}} \int_{u=g_k}^\infty  \Prob\left(S_k \in du ; \tau_g > k \right) e^{-\frac{u^2}{2(n-k)}}.
\end{align*}
It is clear from the above identity that one may use Proposition~\ref{prop:DensityRWConditionedAtTimeK} to obtain the main result. There appears a technical difficulty for the evaluation of the integral within $o(|g_k|)$ distance from the boundary, as Proposition~\ref{prop:DensityRWConditionedAtTimeK} does not provide the asymptotic behavior for these values. Alternatively, we provide an appropriate bound at this interval that is is obtained by using the following lemma.

\begin{lemma}
Suppose $x_k = o(\sqrt{k})$ is such that $x_k=\Omega(|g_k|)$ and $x_k \rightarrow \infty$ as $k \rightarrow \infty$. Then, 
\begin{align*}
\Prob\left( S_k \leq g_k+x_k ; \tau_g>k  \right) \leq (1+o(1))\frac{x_k^2}{\sqrt{2\pi}}  \frac{L_g(k)}{k^{3/2}}.
\end{align*}
\label{lem:CloseToBoundary}
\end{lemma}

\begin{proof}
Note that 
\begin{align*}
\Prob\left( S_k \leq  g_k+x_k ; \tau_g>k  \right) = \Prob\left( \tau_g>k  \right) - \Prob\left( S_k >  g_k+x_k ; \tau_g>k  \right),
\end{align*}
where can determine the latter probability by using Proposition~\ref{prop:DensityRWConditionedAtTimeK}. Note that due to~\eqref{eq:AssUandV}, for any $y_k$ satisfying $y_k=\Omega(|g_k|)$ and $y_k \rightarrow \infty$ as $k \rightarrow \infty$, 
\begin{align*}
U(y_k) \Exp(-\tilde{S}_{T_0}) \sim y_k \geq y_k e^{-\frac{y_k^2}{2k}}.
\end{align*}
Therefore, Proposition~\ref{prop:DensityRWConditionedAtTimeK} yields
\begin{align*}
\Prob\left( S_k >  g_k+x_k | \tau_g>k  \right) &= \int_{t= g_k+x_k}^\infty \Prob\left( S_k \in dt | \tau_g>k  \right) \geq (1-o(1)) \int_{y=x_k}^\infty \frac{y}{k} e^{-\frac{y^2}{2k}} = (1-o(1))e^{-\frac{x_k^2}{2k}}.
\end{align*}
Recalling~\eqref{eq:ConditionedRWTauStoppingTime} and expanding $1-e^{-x_k^2/(2k)}$, we conclude that
\begin{align*}
\Prob\left( S_k \leq  g_k+x_k ; \tau_g>k  \right) \leq (1+o(1)) \sqrt{\frac{2}{\pi}} \frac{L_g(k)}{\sqrt{k}}\left(1-e^{-\frac{x_k^2}{2k}}\right) \leq (1+o(1))\frac{x_k^2}{\sqrt{2\pi}}  \frac{L_g(k)}{k^{3/2}}.
\end{align*}
\end{proof}

\noindent
Next, we prove our main result.

\begin{proof}[Proof of Theorem~\ref{thm:MainResultBoundary}]
Similar as the proof of Proposition~\ref{prop:DensityRWConditionedAtTimeK}, we will provide an appropriate upper and lower bound of our objective, and show that these converge. Fix $\delta>0$, let $g_{k,\delta}=o(\sqrt{n-k})$ be such that $g_{k,\delta} \rightarrow \infty$ as $n \rightarrow \infty$ if $|g_k|$ converges to a finite constant, $g_{k,\delta} = (1-\delta)g_k$ if $-g_k \rightarrow \infty$ as $n \rightarrow \infty$, and $g_{k,\delta} = (1+\delta)g_k$ if $g_k \rightarrow \infty$ as $n \rightarrow \infty$. Then, using~\eqref{eq:RandomWalkConvergenceProperties},
\begin{align*}
\Prob&\left(\tau_g > k | S_n = 0  \right) \leq (1+o(1)) \sqrt{\frac{n}{n-k}} \int_{u=g_k}^\infty \Prob\left(S_k \in du ; \tau_g > k \right) e^{-\frac{u^2}{2(n-k)}}  \\
&\leq (1+o(1)) \sqrt{\frac{n}{n-k}} \left(\int_{u=g_k}^{g_{k,\delta}} \Prob\left(S_k \in du ; \tau_g > k \right) e^{-\frac{u^2}{2(n-k)}} + \int_{u=g_{k,\delta}}^\infty \Prob\left(S_k \in du ; \tau_g > k \right) e^{-\frac{u^2}{2(n-k)}} \right).
\end{align*}
For the first term we obtain the following bounds. If $|g_k|$ converges to a finite constant, Lemma~\ref{lem:CloseToBoundary} yields that 
\begin{align*}
\int_{u=g_k}^{g_{k,\delta}} \Prob\left(S_k \in du ; \tau_g > k \right) e^{-\frac{u^2}{2(n-k)}} \leq \Prob\left(S_k \leq g_{k,\delta}; \tau_g > k \right) = o\left( L_g(k)\frac{n-k}{k^{3/2}} \right).
\end{align*}
If $|g_k| \rightarrow \infty$ when $n\rightarrow \infty$, then Lemma~\ref{lem:CloseToBoundary} yields
\begin{align*}
\int_{u=g_k}^{g_{k,\delta}} \Prob\left(S_k \in du ; \tau_g > k \right) e^{-\frac{u^2}{2(n-k)}} &\leq \Prob\left(S_k \leq g_{k,\delta} ; \tau_g > k \right) e^{-\frac{g_k^2(1-\delta)^2}{2(n-k)}} \leq \frac{1+o(1)}{\sqrt{2\pi}} \cdot  \delta^2 \frac{L_g(k) g_k^2}{k^{3/2}} e^{-\frac{g_k^2(1-\delta)^2}{2(n-k)}}.
\end{align*}
for all $|g_k| = O(\sqrt{n-k})$ and for all $g_k <0$ with $|g_k| =\omega(\sqrt{n-k})$. 

\noindent
For the second term, due to Proposition~\ref{prop:DensityRWConditionedAtTimeK} and~\eqref{eq:AssUandV}, we obtain
\begin{align*}
\int_{u=g_{k,\delta}}^\infty \Prob&\left(S_k \in du ; \tau_g > k \right)  e^{-\frac{u^2}{2(n-k)}} \\
&\leq (1+o(1)) \sqrt{\frac{2}{\pi}} \frac{L_g(k)}{k^{3/2}} \left(\int_{u=g_{k}}^{\infty} (u-g_k) e^{-\frac{u^2}{2(n-k)}} \, du + \int_{u=u_\star}^\infty u e^{-\frac{u^2}{2k}}e^{-\frac{u^2}{2(n-k)}} \right) \, du,
\end{align*}
where $u_\star = o(\sqrt{k})$ is large enough such that $u_\star= \omega(n-k)$. The second term is relatively small, i.e.
\begin{align*}
\int_{u=u_\star}^\infty u e^{-\frac{u^2}{2k}}e^{-\frac{u^2}{2(n-k)}} \, du = \frac{(n-k)k}{n} e^{-\frac{u_\star^2}{2} \frac{n}{k(n-k)}} = o\left( n-k \right).
\end{align*}
On the other hand, due to Lemma~\ref{lem:Integral3},
\begin{align*}
\int_{u=g_{k}}^{\infty} (u-g_k) e^{-\frac{u^2}{2(n-k)}} \, du &= \int_{y=0}^{\infty} y e^{-\frac{(y+g_k)^2}{2(n-k)}}\, dy =\left( (n-k) e^{-\frac{g_k^2}{2(n-k)}}-g_k \sqrt{n-k} \int_{x=g_k/\sqrt{n-k}}^\infty e^{-\frac{x^2}{2}} \, dx \right) \\
&\leq (1+o(1)) \left\{ \begin{array}{ll}
n-k & \textrm{if } |g_k| = o(\sqrt{n-k}) ,\\
(n-k)\gamma(|g_k|/\sqrt{n-k}) & \textrm{if } |g_k| = \Theta(\sqrt{n-k}), \\
\sqrt{2\pi} |g_k| \sqrt{n-k} & \textrm{if } |g_k| = \omega(\sqrt{n-k}) \textrm{ and } g_k<0, \\
\end{array}\right.
\end{align*}
where $\gamma(\cdot)$ is defined as in \eqref{eq:DefEtaGamma}. Hence, we obtain the following upper bounds. If $|g_k|=o(\sqrt{n-k})$,
\begin{align*}
\limsup_{n\rightarrow \infty} \frac{k}{L_g(k)\sqrt{n-k}} \Prob\left(\tau_g > k | S_n = 0  \right) \leq \sqrt{\frac{2}{\pi}}.
\end{align*}
If $|g_k| = \Theta(\sqrt{n-k)}$, define $\eta = \lim_{n \rightarrow \infty} |g_k|/\sqrt{n-k} \in (0,\infty)$. Then,
\begin{align*}
\limsup_{n\rightarrow \infty} \frac{k}{L_g(k)\sqrt{n-k}} \Prob\left(\tau_g > k | S_n = 0  \right) \leq \delta^2 \frac{\eta^2}{\sqrt{2\pi}} e^{-\frac{\eta^2(1-\delta)^2}{2}} + \sqrt{\frac{2}{\pi}} \gamma(\eta) .
\end{align*}
Finally, if $|g_k| = \omega(\sqrt{n-k)}$ and $g_k <0$, then
\begin{align*}
\limsup_{n\rightarrow \infty} \frac{k}{L_g(k)|g_k|} \Prob\left(\tau_g > k | S_n = 0  \right) \leq \sqrt{\frac{2}{\pi}} \cdot \sqrt{2\pi} = 2.
\end{align*}

\noindent
For the lower bound, using Proposition~\ref{prop:DensityRWConditionedAtTimeK}, note that
\begin{align*}
\Prob\left(\tau_g > k | S_n = 0  \right) &\geq (1-o(1)) \sqrt{\frac{n}{n-k}} \int_{u=g_{k,\delta}}^\infty \Prob\left(S_k \in du ; \tau_g > k \right) e^{-\frac{u^2}{2(n-k)}}  \\
&\geq (1-o(1)) \sqrt{\frac{2}{\pi}} \frac{L_g(k)}{k\sqrt{n-k}} \int_{y=\delta |g_k|}^{u^\star}  y e^{-\frac{(y+g_k)^2}{2(n-k)}} \, dy ,
\end{align*}
where $u^\star = o(\sqrt{k})$ is large enough such that $(u^\star+g_k) = \omega(n-k)$. Using Lemma~\ref{lem:Integral3}, we find
\begin{align*}
\int_{y=\delta |g_k|}^{u^\star}  y e^{-\frac{(y+g_k)^2}{2(n-k)}} \, dy &= (n-k)\left(e^{-\frac{(\delta|g_k|+g_k)^2}{2(n-k)}} - e^{-\frac{(u^\star+g_k)^2}{2(n-k)}}\right)-g_k\sqrt{n-k} \int_{x=\delta |g_k|+g_k}^{g_k+u^\star} e^{-\frac{x^2}{2}} \, dx.
\end{align*}
Therefore, if $|g_k| = o(\sqrt{n-k)}$, then
\begin{align*}
\liminf_{n\rightarrow \infty} \frac{k}{L_g(k)\sqrt{n-k}} \Prob\left(\tau_g > k | S_n = 0  \right) \geq \sqrt{\frac{2}{\pi}}.
\end{align*}
If $|g_k| = \Theta(\sqrt{n-k)}$, then
\begin{align*}
\liminf_{n\rightarrow \infty} \frac{k}{L_g(k)\sqrt{n-k}} \Prob\left(\tau_g > k | S_n = 0  \right) \geq \left\{\begin{array}{ll}
\sqrt{\frac{2}{\pi}} \left( e^{-\frac{\eta^2(1+\delta)^2}{2}} - \eta \int_{x=\eta(1 + \delta)}^{\infty} e^{-\frac{x^2}{2}} \right) & \textrm{if } \eta <0,\\
\sqrt{\frac{2}{\pi}} \left( e^{-\frac{\eta^2(1+\delta)^2}{2}} - \eta \int_{x=\eta(1 - \delta)}^{\infty} e^{-\frac{x^2}{2}} \right) & \textrm{if } \eta > 0,
\end{array} \right.
\end{align*}
where $\eta=\lim_{n\rightarrow \infty} |g_k|/\sqrt{n-k}$.
Finally, if $|g_k| = \omega(\sqrt{n-k})$ and $g_k <0$, then
\begin{align*}
\liminf_{n\rightarrow \infty} \frac{k}{L_g(k)|g_k|} \Prob\left(\tau_g > k | S_n = 0  \right) \geq \sqrt{\frac{2}{\pi}} \cdot \sqrt{2\pi} = 2.
\end{align*}
We see that as $\delta \downarrow 0$ all lower bounds coincide with the upper bounds, concluding~\eqref{eq:MainTheoremBoundary}.
\end{proof}

\subsection*{Acknowledgment}
The work in this paper is supported by the Netherlands Organisation for Scientific Research (NWO) \\through Gravitation-grant NETWORKS-024.002.003, and a VICI grant.

\appendix
\section{Proofs of Lemma~\ref{lem:SmallValuesVSmallT} and~\ref{lem:SmallValuesVLargeT}}
\label{app:ProofOfLemmas}
\begin{proof}[Proof of Lemma~\ref{lem:SmallValuesVSmallT}]
Note that
\begin{align*}
&\int_{v= g_{(1-\epsilon)k} }^{g_{(1-\epsilon)k} +{x_\epsilon} \sqrt{k}} \Prob\left(S_{(1-\epsilon)k} \in dv ; \tau_g > (1-\epsilon)k \right) \Prob\left(S_{\epsilon k} \in dt ; T_{g_{k,\epsilon}^+} > \epsilon k | S_0=v \right) \\
&\hspace{0.5cm} \leq  \Prob\left(S_{(1-\epsilon)k} \leq g_{(1-\epsilon)k} +{x_\epsilon} \sqrt{k} ; \tau_g > (1-\epsilon)k \right) \sup_{v \in[g_{(1-\epsilon)k} ,g_{(1-\epsilon)k} +{x_\epsilon} \sqrt{k}]} \Prob\left(S_{\epsilon k} \in dt ; T_{g_{k,\epsilon}^+} > \epsilon k | S_0=v \right).
\end{align*}

\noindent
For $v=o(\sqrt{k})$, Equation~\eqref{eq:Doney1} yields uniformly
\begin{align*}
\frac{\Prob\left(S_{\epsilon k} \in dt ; T_{g_{k,\epsilon}^+} > \epsilon k | S_0=v \right)}{dt} \sim \frac{U(t-g_{k,\epsilon}^+)}{\sqrt{2\pi} (\epsilon k)^{3/2}} V(v-g_{k,\epsilon}^+).
\end{align*}
On the other hand, if $v = \Theta(\sqrt{k})$, then~\eqref{eq:Doney3} yields
\begin{align*}
\frac{\Prob\left(S_{\epsilon k} \in dt ; T_{g_{k,\epsilon}^+} > \epsilon k | S_0=v \right)}{dt} \sim \sqrt{\frac{2}{\pi}} \Exp(-\tilde{S}_{T_0})  \frac{U(t-g_{k,\epsilon}^+)}{ (\epsilon k)^{3/2}} (v-g_{k,\epsilon}^+) e^{-\frac{(v-g_{k,\epsilon}^+)^2}{2\epsilon k}},
\end{align*}
We observe that $e^{-x} \leq 1$ for all $x\geq 0$, and $\Exp(-\tilde{S}_{T_0}) \in (0,\infty)$ since the increments of the random walk have variance one. Moreover, due to~\eqref{eq:AssUandV} and~\eqref{eq:ProductExpectationsLadderHeightsVariance}, there exists a constant $c_1 \in (0,\infty)$ such that 
\begin{align*}
\sup_{v \in[g_{(1-\epsilon)k} ,g_{(1-\epsilon)k} +{x_\epsilon} \sqrt{k}]} \Prob\left(S_{\epsilon k} \in dt ; T_{g_{k,\epsilon}^+} > \epsilon k | S_0=v \right) \leq c_1 (g_{(1-\epsilon)k} +{x_\epsilon} \sqrt{k}-g_{k,\epsilon}^+)\frac{U(t-g_{k,\epsilon}^+)}{ (\epsilon k)^{3/2}} \, dt.
\end{align*}
Due to our assumption~\eqref{eq:AdditionalAssumption}, we have that $(g_{(1-\epsilon)k} -g_{k,\epsilon}^+)<2\alpha(\epsilon)|g_k|=o(\sqrt{k})$. Also, as $U(\cdot)$ is non-decreasing and~\eqref{eq:AssUandV} holds, there exists a constant $c_2 \in (0,1)$ such that
\begin{align*}
U(t-g_{k,\epsilon}^+) \leq c_2(1+\alpha(\epsilon)) U(t-g_k).
\end{align*}
That is, there exists a $c_3 \in (0,\infty)$,
\begin{align*}
\sup_{v \in[g_{(1-\epsilon)k} ,g_{(1-\epsilon)k} +{x_\epsilon} \sqrt{k}]} \Prob\left(S_{\epsilon k} \in dt ; T_{g_{k,\epsilon}^+} > \epsilon k | S_0=v \right) \leq (1+o(1)) c_3 x_\epsilon \, \frac{U(t-g_{k})}{ \epsilon^{3/2} k} \, dt.
\end{align*}

\noindent
Finally, since~\eqref{eq:ConditionedRWDensity} and~\eqref{eq:ConditionedRWTauStoppingTime} hold with $L_g(\cdot)$ a slowly varying function,
\begin{align*}
\Prob&\left(S_{(1-\epsilon)k} \leq g_{(1-\epsilon)k} +{x_\epsilon} \sqrt{k} ; \tau_g > (1-\epsilon)k \right) = \Prob\left(S_{(1-\epsilon)k} \leq g_{(1-\epsilon)k} +{x_\epsilon} \sqrt{k} | \tau_g > (1-\epsilon)k \right) \Prob\left(\tau_g > (1-\epsilon)k \right)\\
&\hspace{4cm}\sim \left(1-e^{-\frac{-{x_\epsilon}^2}{2(1-\epsilon)}}\right) \sqrt{\frac{2}{\pi}}\frac{L_g(k)}{\sqrt{(1-\epsilon)k}} < \sqrt{\frac{2}{\pi}} \frac{\epsilon^{3/2}}{\sqrt{1-\epsilon}} \frac{L_g(k)}{\sqrt{k}}.
\end{align*}

\noindent
Multiplying the final two expressions yields the result.
\end{proof}

\begin{proof}[Proof of Lemma~\ref{lem:SmallValuesVLargeT}]
The proof is similar to the proof of Lemma~\ref{lem:SmallValuesVSmallT}, but in this case we have to consider the asymptotics for $t=\Theta(\sqrt{k})$. Again,
\begin{align*}
&\int_{v= g_{(1-\epsilon)k} }^{g_{(1-\epsilon)k} +{x_\epsilon} \sqrt{k}} \Prob\left(S_{(1-\epsilon)k} \in dv ; \tau_g > (1-\epsilon)k \right) \Prob\left(S_{\epsilon k} \in dt ; T_{g_{k,\epsilon}^+} > \epsilon k | S_0=v \right) \\
&\hspace{0.5cm} \leq  \Prob\left(S_{(1-\epsilon)k} \leq g_{(1-\epsilon)k} +{x_\epsilon} \sqrt{k} ; \tau_g > (1-\epsilon)k \right) \sup_{v \in[g_{(1-\epsilon)k} ,g_{(1-\epsilon)k} +{x_\epsilon} \sqrt{k}]} \Prob\left(S_{\epsilon k} \in dt ; T_{g_{k,\epsilon}^+} > \epsilon k | S_0=v \right).
\end{align*}

\noindent
For $v=o(\sqrt{k})$, Equation~\eqref{eq:Doney2} yields uniformly
\begin{align*}
\frac{\Prob\left(S_{\epsilon k} \in dt ; T_{g_{k,\epsilon}^+} > \epsilon k | S_0=v \right)}{dt} \sim \sqrt{\frac{2}{\pi}} \Exp(-{S}_{T_0})  \frac{V(v-g_{k,\epsilon}^+)}{ (\epsilon k)^{3/2}} (t-g_{k,\epsilon}^+) e^{-\frac{(t-g_{k,\epsilon}^+)^2}{2\epsilon k}}.
\end{align*}
Note $e^{-x} \leq 1$ for all $x\geq 0$ and $g_{k,\epsilon}^+ = o(\sqrt{k})$. Since $V(\cdot)$ is non-decreasing and satisfies~\eqref{eq:AssUandV}, we find that there exists a $c_1 \in (0,\infty)$ such that
\begin{align*}
\sup_{v \in[g_{(1-\epsilon)k} ,g_{(1-\epsilon)k} +{x_\epsilon} \sqrt{k}]} \sqrt{\frac{2}{\pi}} \Exp(-{S}_{T_0})  \frac{V(v-g_{k,\epsilon}^+)}{ (\epsilon k)^{3/2}} (t-g_{k,\epsilon}^+) e^{-\frac{(t-g_{k,\epsilon}^+)^2}{2\epsilon k}} &\leq (1+o(1))c_1 \frac{x_\epsilon \sqrt{k}}{ (\epsilon k)^{3/2}} t e^{-\frac{t^2}{2 k}}e^{-\frac{(1-\epsilon) t^2}{2\epsilon k}} \\
&\leq (1+o(1))c_1 \frac{x_\epsilon}{\epsilon^{3/2} k} t e^{-\frac{t^2}{2 k}}.
\end{align*}

\noindent
On the other hand, if $v = \Theta(\sqrt{k})$, then~\eqref{eq:Doney3} yields
\begin{align*}
\frac{\Prob\left(S_{\epsilon k} \in dt ; T_{g_{k,\epsilon}^+} > \epsilon k | S_0=v \right)}{dt} \sim \frac{1}{\sqrt{2 \pi \epsilon k}} \left( e^{-\frac{(v-t)^2}{2\epsilon k}} - e^{-\frac{(v+t-2g_{k,\epsilon}^+)^2}{2\epsilon k}} \right) \sim \frac{1}{\sqrt{2 \pi \epsilon k}} \left( e^{-\frac{(v-t)^2}{2\epsilon k}} - e^{-\frac{(v+t)^2}{2\epsilon k}} \right),
\end{align*}
Using a Taylor expansion, we obtain
\begin{align*}
\left( e^{-\frac{(v-t)^2}{2\epsilon k}} - e^{-\frac{(v+t)^2}{2\epsilon k}} \right)= e^{-\frac{v^2}{2\epsilon k}-\frac{t^2}{2\epsilon k}}\left( e^{\frac{vt}{\epsilon k}} - e^{-\frac{vt}{2\epsilon k}} \right)\leq e^{-\frac{t^2}{2 k}} \left(2\frac{vt}{\epsilon k} + o\left(\frac{vt}{\epsilon k} \right) \right).
\end{align*}
Therefore there exists a $c_2\in (0,\infty)$ such that
\begin{align*}
\sup_{v \in[g_{(1-\epsilon)k} ,g_{(1-\epsilon)k} +{x_\epsilon} \sqrt{k}]} \frac{1}{\sqrt{2 \pi \epsilon k}} \left( e^{-\frac{(v-t)^2}{2\epsilon k}} - e^{-\frac{(v+t)^2}{2\epsilon k}} \right) \leq (1+o(1)) c_2 \frac{x_\epsilon}{\epsilon^{3/2} k} t e^{-\frac{t^2}{2 k}} . 
\end{align*}
We can conclude that there must exist a $c_3 \in (0,\infty)$ such that
\begin{align*}
\sup_{v \in[g_{(1-\epsilon)k} ,g_{(1-\epsilon)k} +{x_\epsilon} \sqrt{k}]} \Prob\left(S_{\epsilon k} \in dt ; T_{g_{k,\epsilon}^+} > \epsilon k | S_0=v \right) &\leq (1+o(1)) c_3  \frac{x_\epsilon}{\epsilon^{3/2} k} t e^{-\frac{t^2}{2 k}}.  
\end{align*}
Again, since~\eqref{eq:ConditionedRWDensity} and~\eqref{eq:ConditionedRWTauStoppingTime} hold with $L_g(\cdot)$ a slowly varying function,
\begin{align*}
\Prob&\left(S_{(1-\epsilon)k} \leq g_{(1-\epsilon)k} +{x_\epsilon} \sqrt{k} ; \tau_g > (1-\epsilon)k \right) \sim \left(1-e^{-\frac{-{x_\epsilon}^2}{2(1-\epsilon)}}\right) \sqrt{\frac{2}{\pi}}\frac{L_g(k)}{\sqrt{(1-\epsilon)k}} < \sqrt{\frac{2}{\pi}} \frac{\epsilon^{3/2}}{\sqrt{1-\epsilon}} \frac{L_g(k)}{\sqrt{k}}.
\end{align*} 
Multiplying this with the previous expression then concludes the proof.
\end{proof}

\section{Useful integral identities}
\begin{lemma}
Suppose $\epsilon \in (0,1)$. For every $x \geq 0$,
\begin{align*}
\int_{y=x }^\infty y^2 e^{-\frac{y^2}{2\epsilon(1-\epsilon)}} \, dy = x \epsilon (1-\epsilon) e^{-\frac{x^2}{2\epsilon(1-\epsilon)}} + (\epsilon(1-\epsilon))^{3/2} \int_{y=x/\sqrt{\epsilon(1-\epsilon)}}^\infty e^{-y^2/2} \, dy.
\end{align*}
\label{lem:Integral1VSmall}
\end{lemma}

\begin{proof}
This follows directly from an integration by parts and a variable substitution.
\end{proof}

\begin{lemma}
Suppose $\epsilon \in (0,1)$ and $c \in (0,\infty)$. Then,
\begin{align*}
\int_{y= 0}^\infty & \frac{y}{(1-\epsilon)} e^{-\frac{y^2}{2(1-\epsilon)}}\left(e^{-\frac{(y-c)^2}{2\epsilon}}-e^{-\frac{(y+c)^2}{2\epsilon }} \right) \, dy = \sqrt{2\pi} \sqrt{\epsilon(1-\epsilon)} c e^{-\frac{c^2}{2}}.
\end{align*}
Moreover, for every $x \geq 0$, 
\begin{align*}
\int_{y= 0}^x & \frac{y}{(1-\epsilon)} e^{-\frac{y^2}{2(1-\epsilon)}}\left(e^{-\frac{(y-c)^2}{2\epsilon}}-e^{-\frac{(y+c)^2}{2\epsilon }} \right) \, dy \leq 1-e^{-\frac{x^2}{2(1-\epsilon)}}.
\end{align*}
\label{lem:Integral2VLarge}
\end{lemma}

\begin{proof}
First note that for every $a,b \in \mathbb{R}$,
\begin{align*}
\int y e^{-\frac{y^2}{2a}+\frac{y}{b}} \, dy = -a e^{-\frac{y^2}{2a}+\frac{y}{b}} - \frac{a^{3/2}}{b} e^{\frac{a}{2 b^2}} \int_{s=0}^{\frac{a-by}{\sqrt{a}b}} e^{-\frac{s^2}{2}} \, ds.
\end{align*}
Therefore,
\begin{align*}
\int_{y= 0}^\infty & \frac{y}{(1-\epsilon)} e^{-\frac{y^2}{2(1-\epsilon)}}\left(e^{-\frac{(y-c)^2}{2\epsilon}}-e^{-\frac{(y+c)^2}{2\epsilon }} \right) \, dy =e^{-\frac{c^2}{2\epsilon}}\int_{y= 0}^\infty  \frac{y}{(1-\epsilon)} e^{-\frac{y^2}{2\epsilon(1-\epsilon)}}\left(e^{\frac{yc}{\epsilon}}-e^{-\frac{yc}{\epsilon }} \right) \, dy \\
&=\frac{1}{1-\epsilon} e^{-\frac{c^2}{2\epsilon}} \cdot \sqrt{2\pi} \epsilon^{1/2} (1-\epsilon)^{3/2} c e^{\frac{c^2}{2} \frac{1-\epsilon}{\epsilon}} = \sqrt{2\pi} \sqrt{\epsilon(1-\epsilon)} c e^{-\frac{c^2}{2}}.
\end{align*}

The second claim follows easily by observing that for every $y \in \mathbb{R}$,
\begin{align*}
e^{-\frac{(y-c)^2}{2\epsilon}}-e^{-\frac{(y+c)^2}{2\epsilon }} \leq 1.
\end{align*}
\end{proof}

\begin{lemma}
Let $a,b \in \mathbb{R}$ be such that $0\leq a \leq b \leq \infty$. Then
\begin{align*}
\int_{y=a}^{b} y e^{-\frac{(y+g)^2}{2(n-k)}} \, dy = (n-k) \left(e^{-\frac{(a+g_k)^2}{2(n-k)}}-e^{-\frac{(b+g_k)^2}{2(n-k)}} \right)-g_k\sqrt{n-k}\int_{x=(a+g_k)/\sqrt{n-k}}^{(b+g_k)/\sqrt{n-k}} e^{-\frac{x^2}{2}} \, dx.
\end{align*}
In particular,
\begin{align*}
\int_{y=0}^{\infty} y e^{-\frac{(y+g)^2}{2(n-k)}} \, dy = (n-k) e^{-\frac{g_k^2}{2(n-k)}} -g_k\sqrt{n-k}\int_{x=g_k/\sqrt{n-k}}^{\infty} e^{-\frac{x^2}{2}} \, dx.
\end{align*}
\label{lem:Integral3}
\end{lemma}

\begin{proof}
This is an easy consequence of a variable substitution.
\end{proof}

\bibliographystyle{plain}

\end{document}